\documentclass{birkjour}

\usepackage{amsmath}
\usepackage{graphicx} 
  \usepackage{epsfig}
  \usepackage{bbm}
\usepackage{epstopdf}
\usepackage{color}
\usepackage{xfrac}
\usepackage{lmodern}
\usepackage{fix-cm}
\usepackage{listings}
\usepackage{fancyvrb}
\usepackage{enumerate}   
\usepackage{xcolor}
\usepackage{pgf,tikz}
\usepackage{mathrsfs}
\usetikzlibrary{arrows}
\usepackage{float}
\usepackage{multicol}
\usepackage{lscape}
\usepackage{rotating}

\usetikzlibrary{decorations.markings}
\usepackage{subcaption}

 \theoremstyle{definition}
 
 \theoremstyle{remark}

 \numberwithin{equation}{section}

\newtheorem{theorem}{Theorem}[section]

\newtheorem{lemma}[theorem]{Lemma}
\newtheorem{proposition}{Proposition}

\theoremstyle{definition}
\newtheorem{definition}[theorem]{Definition}
\newtheorem{remark}{Remark}

\DeclareMathOperator{\PP}{\mathbb{P}}
\DeclareMathOperator{\Q}{\mathbb{Q}}
\DeclareMathOperator{\A}{\mathcal{A}}
\DeclareMathOperator{\T}{\mathcal{T}}

\newcommand{\I}{{\mathcal I}}

\newcommand{\lra}{\longrightarrow}
\newcommand{\ra}{\rightarrow}

\newcommand{\cB}{{\mathcal B}}
\newcommand{\cG}{{\mathcal G}}
\newcommand{\cE}{{\mathcal E}}
\newcommand{\cI}{{\mathcal I}}

\newcommand{\cA}{{\mathcal A}}

\newcommand{\cS}{{\mathcal S}}
\newcommand{\cL}{{\mathcal L}}

\newcommand{\cC}{{\mathcal C}}
\newcommand{\cH}{{\mathcal H}}

\newcommand{\cU}{{\mathcal U}}

\newcommand{\cT}{{\mathcal T}}

\newcommand{\cW}{{\mathcal W}}

\renewcommand{\phi}{\varphi}
\renewcommand{\a}{{\alpha}}

\renewcommand{\l}{\lambda}

\newcommand{\Tau}{\mathrm{T}}

\newcommand{\supp}{\mathrm{supp}\;}

\newcommand{\Lip}{\operatorname{Lip}}

\newcommand{\R}{{\mathbb{R}}}

\allowdisplaybreaks
\begin{document}
%
%
%
%
%
%
%
%
%

\title[Hybrid control and optimal visiting]
 {Hybrid control for optimal visiting problems for a single player and for a crowd}

\author{Fabio Bagagiolo}
\address{Dipartimento di Matematica \\
Universit\`a di Trento\\
Via Sommarive, 14, 38123 Povo (TN) Italy}
\email{fabio.bagagiolo@unitn.it }

\author{Adriano Festa}
\address{Dipartimento di Scienze Matematiche ``Giuseppe Luigi Lagrange"\\
Politecnico di Torino\\
Corso Duca degli Abruzzi, 24, 10129 Torino Italy}
\email{adriano.festa@polito.it}
\author[Luciano Marzufero]{Luciano Marzufero*}

\address{%
Dipartimento di Matematica \\
Universit\`a di Trento\\
Via Sommarive, 14, 38123 Povo (TN) Italy}

\email{luciano.marzufero@unitn.it }



\date{\today}

\begin{abstract}
In an optimal visiting problem, we want to control a trajectory that has to pass as close as possible to a collection of target points or regions. We introduce a hybrid control-based approach for the classic problem where the trajectory can switch between a group of discrete states related to the targets of the problem. The model is subsequently adapted to a mean-field game framework to study viability and crowd fluxes to model a multitude of indistinguishable players. 
\end{abstract}

\maketitle
\section{Introduction}

In this paper, we deal with the problem of optimizing a trajectory to ``visit'', i.e., to touch or at least to pass as close as possible, to a collection of targets. In the following, we refer to this problem as \emph{optimal visiting}. The issue presents various inherent difficulties: some related to its high computational complexity (shared with other well-known optimization problems as the ``Traveling salesman problem'' \cite{gavish1978travelling}) and other related to its possible continuous/discontinuous nature.
\par
Let us state the problem more precisely:  consider the controlled dynamics
\begin{equation}
\label{dyn}
\begin{cases}
y'(s)=f(y(s),\alpha(s),s),&s\in]t,T]\\
y(t)= x,&x\in\R^d
\end{cases},
\end{equation}
where $t\in[0, T]$, $\alpha:[t,T]\longrightarrow A$ is a measurable control function, and the dynamics $f:\R^d\times A\times[0,+\infty[\longrightarrow \R^d$ is suitably regular. Consider a collection of $N$ compact disjoint target sets in $\R^d$, $\{\T_1,\T_2,\ldots,\T_N\}$, and $y_{(x,t)}(\cdot; \a)$ a solution of \eqref{dyn} related to a starting point $x$, a starting time $t$ and a control $\alpha$. We can write the optimal visiting problem just considering, for example, the minimization of a cost functional of the form, in a finite horizon feature,
$$
J(x,t,\alpha)=\int_t^T\ell(y(s),\alpha(s),s)ds,
$$
with the running cost $\ell$ suitably designed in order to keep trace of the distances from the targets of the trajectory. The visiting cost is then defined as
$$
v(x,t)=\inf_{\alpha\in\A}J(x,t,\alpha).
$$
Actually, the problem requires a particular framework as a standard continuous optimal control setting fails to describe the problem correctly. Let us illustrate this difficulty using the following toy example.
\par\smallskip
Let us consider the $1$-D problem with $A=\R$, $f(x,a)=a$, and $\cT_1=(-\infty,-1]$, $\cT_2=[1,+\infty)$. We focus on designing an optimal visit formulation in the interval $[-1,1]$. 
At first, we consider the easiest running cost design:
$$
\ell(y,a,t):= \frac{1}{2}\left(\sum_{j=1}^N d(y, \T_j)^2+\|a\|^2\right),
$$
which penalizes quadratically the distance from the targets and the norm of the control. It is easy to verify that we have a feedback formula for the optimal control as 
$$
\alpha(x,t)=-\frac{x}{(1-t)^2+1/2},
$$
which means that the trajectory is led to zero, which is the middle point between the two targets. Since we want to model a slightly different problem, i.e., a visit more than a compromise between the distances, we are unsatisfied with this result. 
\par
Therefore, we should include the information about the visit of the targets in a different way in the model, allowing us to focus on a single target, as well as on a subfamily of targets, at once. If, for example, we consider the problem of visiting first target $\T_1$ and then $\T_2$, we can easily observe that we would obtain a different problem just swapping the order of the visit. This is a consequence because no Dynamical Programming Principle would be available for the function $v(x, t)$, since the only information brought by the state-position $x$ does not give information about the already visited targets (see also \cite{bagben}). Hence, at this level, we can not in general characterize the optimal visiting time as a solution of a Hamilton-Jacobi-Bellman equation. Consequently, it is quite challenging to perform a global study of the problem or obtain a feedback optimal control map. 
\par\smallskip
The argument above suggests that we need to include in the model a ``memory" of the targets already visited. The latter can be done using various tools. Here, we opt for a hybrid control based construction. In particular, to each target $\T_j$ we associate a label $p^j\in\{0,1\}$. These labels indicate whether the corresponding target has been already visited $p^j=1$, or not $p^j=0$, and they have a discontinuous evolution in time. We then split the optimal visiting problem into several problems, labeled by the $N$-strings $p\in\{0,1\}^N$, and we suitably interpret it as a collection of several optimal stopping/switching problems coupled to each other by the stopping/switching cost (a switch between $N$-strings corresponds to the visit of a target or the choice to forgo visiting one or more targets). 
\par\smallskip
In the following, we adapt to our framework some classic results of viscosity solutions theory that can be found, e.g.,  in \cite{BCD97, Festa2017127}. In particular, the hybrid framework that we propose is related to hybrid control \cite{Branicky199831} and somehow to 
the mathematical switching hysteresis models \cite{Visintin20061}. The need of memory feature, associated with the optimal visiting and dynamic programming and Hamilton-Jacobi equations has been presented in \cite{bagben}, where a continuous hysteresis memory was introduced. The use of a switching/discontinuous/hybrid memory, as in the present paper, was instead used for a one-dimensional optimal visiting problem on a network in \cite{bagfagmagpes}. For switching hybrid control problems, as well as for differential games, related to the model here presented, and in connection with Hamilton-Jacobi equations, we refer to \cite{bagdan} and to \cite{bagmagzop}. A more general discussion is done in \cite{bensoussan1997hybrid} (similar formulations for the deterministic case have also been proposed in \cite{Branicky199831, dharmatti2005hybrid}). 
\par\smallskip
The literature concerning the Traveling Salesman Problem, to which our optimal visiting problem is related, is very large. We only quote an early paper by R. Bellman \cite{bellman1962dynamic} devoted to the problem and dynamic programming.
\par\smallskip
In this paper, we are also interested in the case where a huge population of agents plays the optimal visiting problem with controlled dynamics and costs also depending on the distribution of the population. Actually, the study of the interacting motion of many agents with more than one target seems to be rather new in the literature, especially for what concerns the corresponding continuity equation for the mass distribution. In this work, we start such a kind of study and we furnish some numerical promising results and the analysis for some continuity equations with a mass-sink (corresponding to the case where some agents, possibly labeled by the same $N$-strings $p$, forgo visiting some targets and then pass to another level, labeled by another $N$-string $p'$).
\par
The model for a crowd of indistinguishable players is taken from the framework of mean-field games \cite{ll3,Huang1,gomessurvey13,cardaliaguet2015mean}, while the adaptation of the same structure to hybrid processes has been only very recently attempted. Some related works are \cite{bertucci2018optimal, bertucci2020}, where the author discusses a mean-field optimal stopping problem and \cite{festa2018mean}, where a hybrid mean-field game is presented to model a multi-lane traffic flux of vehicles. 
\par
As it is mentioned later in the paper, some complementary results, related to a single player problem, are referred to \cite{BFMProc1}, while in \cite{BFMProc2} the same framework is used to solve a series of applied problem arising from the sport of orienteering races.
\par\smallskip
The article is organized as follows. In Section \ref{s-single}, we introduce the optimal visiting problem of a single agent, shortly reporting all the theoretical elements that justify the use of a Hamilton-Jacobi formulation. A numerical scheme to approximate the solution is described in Section \ref{s-test1}, with a test to verify in practice the model. In Section \ref{s-crowd}, which contains the main results of the paper,  we consider a crowd of indistinguishable players focusing on the good position of the continuity equation that models the motion of the density of players. Then, in Section \ref{s-test2}, we introduce a numerical approximation for the continuity equation and we examine the model in action through a collection of simple tests.

\section{The optimal visiting problem}
\label{s-single}

We are given of a collection of $N$ disjoint compact subsets $\{\cT_j\}_{j=1,\ldots,N}\subset\R^d$ and we represent the state of the system by $(x, p)\in\R^d\times\I$, where $p=(p^1, p^2,\ldots, p^N)\in\cI=\{0,1\}^N$. Hence $x$ is the continuous state variable and $p$ is the switching discrete state variable. The evolution of the continuous variable is described by the controlled dynamics
\begin{equation} \label{eq_stato}
\begin{cases}
 y'(s)=f(y(s),\alpha(s), q(s)),\ s\in ]t,T]\\
 y(t)= x,\ q(t)=p
\end{cases},
\end{equation}
where $(x, p)\in\R^d\times\cI$ is the initial state, $t\in[0, T]$ the initial instant, $T>0$ the finite horizon. The measurable control is (for $A\subset\mathbb{R}^m$ compact)
$$
\alpha\in{\cA}:=\left\{\alpha:[0,+\infty[\lra A\ \mbox{measurable}\right\},
$$
and the dynamics $q(\cdot)$ of the switching variable is subject to
\begin{equation}
\label{leggeq}
\exists\tau\in[t, s], \ y(\tau)\in{\T_j}\Rightarrow\ q^j(s)=1;\ \ q^j(s)=p^j\ \text{otherwise}.
\end{equation}
In particular, $q^j(s)=0$ means that the target $\cT_j$ has not been visited yet in $[t, s]$ and viceversa for $q^j(s)=1$. The dynamics $f :\R^d\times A \times \I\longrightarrow \R^d$ is continuous,  bounded and Lipschitz continuous w.r.t. $x\in\R^d$ uniformly w.r.t. $(a,p)\in A\times\cI$, i.e., there exists $L>0$ such that
$$
\|f(x,a,q)-f(y,a,q)\|\le L\|x-y\| \ \ \text{for all}\ (x,y)\in\R^d\ \text{and} \ (a,q)\in A\times\cI.
$$
The state of the system at time $s$ is the pair $(y(s), q(s))$ and, for every initial state $(x,t,p)$ and control $\alpha$, by our hypotheses the existence of a unique solution $(y(s), q(s))$ of \eqref{eq_stato}-\eqref{leggeq} is guaranteed. In particular, note that the number of switches of the variable $q$ is necessarily finite. 
\par\smallskip
The optimal visiting problem consists then in reaching the discrete state $\bar p=(1,1,\dots,1)$ (i.e., to visit all the targets) at a time $t\leq \bar t\leq T$, minimizing the following cost
$$
\int_t^{\bar t}e^{-\l(s-t)}\ell(y(s), \a(s), q(s), s)ds,
$$
for a given running cost $\ell$ and a discount factor $\l>0$. 
\subsection{A hybrid-control relaxation of the problem: optimal switching}
\label{hybconrel}
The optimal control problem described above requires to ``exactly touch" all the targets. This makes the evolution of the discrete variable $q$ rather complicated, in particular in view of the corresponding Hamilton-Jacobi equation. We then relax the problem asking instead for ``to pass as close as possible" to each target. Then we assume that we can definitely get rid of some targets and take into account only the remaining ones. In doing that, we also pay an additional cost depending, for instance, on the actual distance from the discarded targets. In this way, the evolution $q(\cdot)$ of the discrete variables is no more a solution of \eqref{eq_stato}-\eqref{leggeq}, but instead, it becomes a control at our disposal. Obviously, there are some constraints: for example, for $N=3$, if $p=(1, 0 ,0)$, $p'=(1, 0, 1)$, $p''=(0, 1, 1)$ and $p'''=\bar p=(1, 1, 1)$, then from $p$ we can not switch to $p''$ otherwise we lose the information about the already visited/discarded target $\cT_1$. However, we can switch to $p'''$ directly.
The process above is sketched in the Figure \ref{switch}. In particular, by an optimization criterium, such a process is feasible because at every switching instant we get rid of a maximal quantity of targets, and hence no infinitesimal accumulation of subsequent switches is possible (no Zeno phenomenon). See also Figure \ref{switch2}.

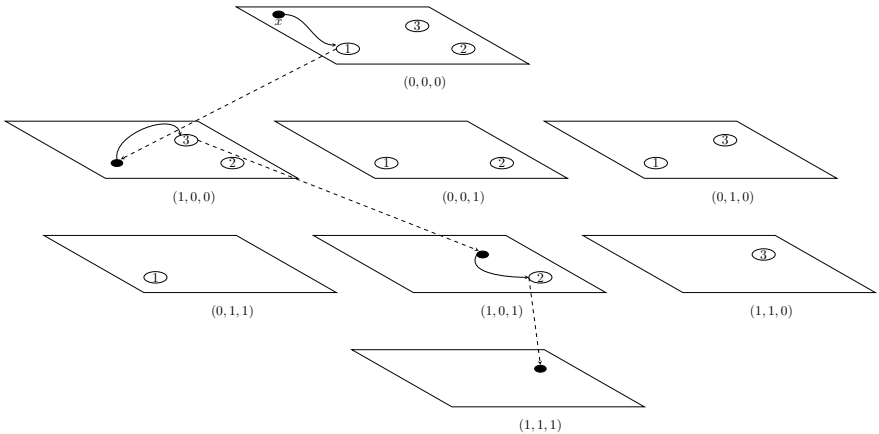
\begin{figure}[t]
\centering
\resizebox{!}{0.49\textwidth}{
\begin{tikzpicture}
\draw (0,0)--(5,0)--(5-2.598,1.5)--(0-2.598,1.5)--cycle;
\draw (0-1,0+3)--(5-1,0+3)--(5-2.598-1,1.5+3)--(0-1-2.598,1.5+3)--cycle;
\draw (0-1-7,0+3)--(5-1-7,0+3)--(5-2.598-7-1,1.5+3)--(0-1-2.598-7,1.5+3)--cycle;
\draw (0-1+7,0+3)--(5-1+7,0+3)--(5-2.598+7-1,1.5+3)--(0-1-2.598+7,1.5+3)--cycle;
\draw (0-1-1,0+3+3)--(5-1-1,0+3+3)--(5-2.598-1-1,1.5+3+3)--(0-1-2.598-1,1.5+3+3)--cycle;
\draw (0-1-7-1,0+3+3)--(5-1-7-1,0+3+3)--(5-2.598-7-1-1,1.5+3+3)--(0-1-2.598-7-1,1.5+3+3)--cycle;
\draw (0-1+7-1,0+3+3)--(5-1+7-1,0+3+3)--(5-2.598+7-1-1,1.5+3+3)--(0-1-2.598+7-1,1.5+3+3)--cycle;
\draw (0-1-1-1,0+3+3+3)--(5-1-1-1,0+3+3+3)--(5-2.598-1-1-1,1.5+3+3+3)--(0-1-2.598-1-1,1.5+3+3+3)--cycle;

\node at (2.3, -0.5){$(1,1,1)$};
\node at (2.3-1, -0.5+3){$(1,0,1)$};
\node at (2.3-1-7, -0.5+3){$(0,1,1)$};
\node at (2.3-1+7, -0.5+3){$(1,1,0)$};
\node at (2.3-1-1, -0.5+3+3){$(0,0,1)$};
\node at (2.3-1-7-1,-0.5+3+3){$(1,0,0)$};
\node at (2.3-1+7-1, -0.5+3+3){$(0,1,0)$};
\node at (2.3-1-1-1, -0.5+3+3+3){$(0,0,0)$};

\draw (0-1-1-1+0.3,0+3+3+3+0.4) ellipse (0.3 cm and 0.15 cm);
\node at  (0-1-1-1+0.3,0+3+3+3+0.4){$1$};
\draw (0-1-1+0.3,0+3+3+0.4) ellipse (0.3 cm and 0.15 cm);
\node at  (0-1-1+0.3,0+3+3+0.4){$1$};
\draw (0-1-1+7+0.3,0+3+3+0.4) ellipse (0.3 cm and 0.15 cm);
\node at  (0-1-1+7+0.3,0+3+3+0.4){$1$};
\draw (0-1-7+0.3,0+3+0.4) ellipse (0.3 cm and 0.15 cm);
\node at  (0-1-7+0.3,0+3+0.4){$1$};

\draw (0-1-1-1+0.3+3,0+3+3+3+0.4) ellipse (0.3 cm and 0.15 cm);
\node at  (0-1-1-1+0.3+3,0+3+3+3+0.4){$2$};
\draw (0-1-1+0.3+3,0+3+3+0.4) ellipse (0.3 cm and 0.15 cm);
\node at  (0-1-1+0.3+3,0+3+3+0.4){$2$};
\draw (0-1+0.3+3,0+3+0.4) ellipse (0.3 cm and 0.15 cm);
\node at  (0-1+0.3+3,0+3+0.4){$2$};
\draw (0-1-1-7+0.3+3,0+3+3+0.4) ellipse (0.3 cm and 0.15 cm);
\node at  (0-1-1-7+0.3+3,0+3+3+0.4){$2$};

\draw (0-1-1-1+0.3+1.8,0+3+3+3+1) ellipse (0.3 cm and 0.15 cm);
\node at  (0-1-1-1+0.3+1.8,0+3+3+3+1){$3$};
\draw (0-1-1-7+0.3+1.8,0+3+3+1) ellipse (0.3 cm and 0.15 cm);
\node at  (0-1-1-7+0.3+1.8,0+3+3+1){$3$};
\draw (0-1-1+7+0.3+1.8,0+3+3+1) ellipse (0.3 cm and 0.15 cm);
\node at  (0-1-1+7+0.3+1.8,0+3+3+1){$3$};
\draw (0-1+7+0.3+1.8,0+3+1) ellipse (0.3 cm and 0.15 cm);
\node at  (0-1+7+0.3+1.8,0+3+1){$3$};
\filldraw (0-1-1-1-1.5,0+3+3+3+1.3) ellipse (0.15 cm and 0.09 cm)node[anchor=north]{$x$};
\filldraw (0-1-1+0.3-7,0+3+3+0.4) ellipse (0.15 cm and 0.09 cm);
\filldraw (0-1-1+0.3+2.5,0+3+1) ellipse (0.15 cm and 0.09 cm);
\filldraw (0-1-1+0.3+4,0+1) ellipse (0.15 cm and 0.09 cm);

\draw[dashed,-stealth] (0-1-1-1,0+3+3+3+0.4)--(0-1-1+0.3-7+0.1,0+3+3+0.4+0.1) ;
\draw[dashed,-stealth](0-1-1-7+0.3+1.8+0.3,0+3+3+1)-- (0-1-1+0.3-0.1+2.5,0+3+1+0.1);
\draw[dashed,-stealth] (0-1+3,0+3+0.4)-- (0-1-1+0.3+4,0+1+0.1);

\draw[-stealth] (0-1-1+0.3-0.15+2.5,0+3+1)  to[out=-120,in=180]  (0-1+3,0+3+0.4);
\draw[-stealth] (0-1-1-1-1.5+0.15,0+3+3+3+1.3) to[out=0,in=180]   (0-1-1-1,0+3+3+3+0.4+0.1);
\draw[-stealth] (0-1-1+0.3-7,0+3+3+0.4+0.1)to[out=100,in=100](0-1-1-7+0.3+1.8-0.15,0+3+3+1+0.11);
\end{tikzpicture}}
\caption{An optimal visiting problem with three targets: $1$, $2$, $3$. The initial state is $(x, (0, 0, 0))$: no target visited/discarded yet. The agent first visits/discards target 1 and then the label switches to $(1, 0, 0)$. The second visited/discarded target is 3, and hence the second switch is to $(1, 0, 1)$. After visiting/discarding target 2, the final switch is to $(1, 1, 1)$. The rectangular indicates $\R^d$ and, for every label, the corresponding already visited/discarded targets are not displayed.}\label{switch}
\end{figure}
\smallskip

\par
Therefore, for any $p$, we denote by $\cI_p$ the set of all possible new variables in $\cI$ after a switch from $p$:
\begin{multline*}
\cI_p=\{\tilde p\in\cI:p^i=1 \Rightarrow \tilde p^i=1\text{ and}\ \exists l=1,\ldots,N:p^l=0, \; \tilde p^l=1\}.
\end{multline*}
Note that in particular $\cI_{\bar p}=\emptyset$, where $\bar p=(1, 1, \ldots1)$. 
\par\smallskip
For a given $p$, the number of the admissible subsequent switches is at most $N-\sum_{i}p^i\leq N$. Given the state $(x, p)$ at the time $t$ with $p\neq\bar p$, the controller chooses: the measurable control $\a\in\cA$, and the discrete one $q:[0, +\infty[\lra\cI$ which contains: the number $1\leq m\leq N-\sum_{i}p^i$ of switches to be performed in order to reach $\bar p$, the switching instants $t\leq t_1<t_2<\ldots<t_m\leq T$ and the switching destinations $p_1,\ldots,p_{m-1}$, $p_m=\bar p$. Such destinations must satisfy $p_1\in\cI_p$, $p_{i+1}\in\cI_{p_i}$, $i=1,\ldots,m-1$. To resume, the control at disposal is then
$$
(\a, m, t_1,\ldots,t_m, p_1,\ldots,p_{m-1})=:u,
$$
and note that for any $(x, p, t)$ as above such a string belongs to a set depending on $p$ and $t$ denoted by $\cU_{(p, t)}$. The cost to be minimized is
\begin{multline*}
J(x, t, p, u)=\sum_{j=1}^m\Bigg(\int_{t_{j-1}}^{t_j}e^{-\l(s-t)}\ell(y(s), \a(s), p_{j-1}, s)ds\\
+e^{-\l(t_j-t)}C(y(t_j), p_{j-1}, p_j)\Bigg),
\end{multline*}
with $\l\ge0$, $p_0=p$, $t_0=t$ and $y(s)$ is the solution of \eqref{eq_stato} where $q(s)=p_{j-1}$ if $s\in[t_{j-1}, t_j]$. 
\par
We assume $\ell:\R^d\times A\times\cI\times[0, T]\lra[0, +\infty[$ bounded, continuous and uniformly continuous w.r.t. $x$ uniformly w.r.t. $a\in A$, $p\in\cI$ and $t\in[0, T]$.
Moreover $C:\R^d\times\cI\times\cI\lra[0, +\infty[$ is uniformly continuous w.r.t. $x\in\R^d$, uniformly w.r.t. $p, p'\in\cI\times\cI_p$. Note that $C(x, p, p')$ represents the switching cost from $p$ to $p'$ when the state position is $x\in\R^d$. For example, it may depend on the distance from the discarded targets, that is $C(x, p, p')=\sum_j\chi_j(p, p')d(x, \cT_j)$, where
$$
\chi_j(p, p')=\begin{cases}0,&p^j=p'^j\\
1,&\text{otherwise}
\end{cases}.
$$
The value function of the problem is
\begin{equation}
\label{eq:V-switching}
V(x, t, p)=\inf_{u\in \cU_{(p, t)}}J(x, t, p, u).
\end{equation}

\subsection{Another possible interpretation: a family of optimal stopping problems}
Our aim is to make the optimal switching problem of the previous subsection more prone to be solved by an algorithmic procedure using Hamilton-Jacobi type problems. We then introduce a possible formulation as a family of time-dependent optimal stopping subproblems, one per every switching variable $p$, suitably coupled by the stopping costs. To do this, we proceed with a backward construction, whose details are explained in \cite{BFMProc1}. Anyway, we give here a short exposure. We start from the last switching variable, that is, the variable $p$ such that $\sum_{i}p^i=N-1$ (i.e., from $p$ we can switch only to $\bar p$). At every step, we get a time-dependent optimal stopping problem in the state space $\R^d$ where, for a given $(x, t)$, the admissible controls are the triples $u=(\a\in\cA, \tau\in[t,T], p'\in\cI_p)$ and the cost to be minimized is
\begin{multline*}
J_p(x, t, \a, \tau, p')=\int_t^{\tau}e^{-\l(s-t)}\ell(y(s), \a(s), p, s)ds
\\
+e^{-\l(\tau-t)}\Big(C(y(\tau), p, p')+V_{p'}(y(\tau), \tau)\Big).
\end{multline*}
The value function is
\begin{equation}
\label{funzionivaloritimep}
V_p(x, t)=\inf_{(\a, \tau, p')}J_p(x, t, \a, \tau, p').
\end{equation}
Note that, if $p$, for instance, is such that $\sum_{i}p^i=N-2$, then from $p'$ we can only switch to the final state $\bar p$, and hence $V_{p'}$ can be a priori evaluated as in the previous step of the procedure. Since when $p=\bar p$, the game stops, we set $V_{\bar p}\equiv0$. Therefore, proceeding backwardly in this way, we can at least formally compute the value functions $V_p$ for any $p\in\cI$. 

\subsection{Equivalence of the two models and characterization of the value function}
We show the equivalence between the optimal switching problem and the family of the optimal stopping ones, i.e.,  $V(x, t, p)=V_p(x, t)$ for every $(x, t, p)\in\R^d\times[0, T]\times\cI$. Here, and in the sequel, $V$ is the value function defined in \eqref{eq:V-switching} and $V_p$ is the value function defined backwardly as in \eqref{funzionivaloritimep}. For the proofs of all the results in this section, we refer to \cite{BFMProc1}. 
\begin{proposition}
\label{dynprog}
Under the hypotheses of \S\ref{hybconrel}, we have that
\begin{itemize}
\item[$(i)$]$V$ and $V_p$ are bounded and uniformly continuous for every $p\in\cI$;
\item[$(ii)$] for every $x\in\mathbb{R}^d$, $t\in[0, T]$ and $p\in\cI$,
\begin{multline*}
V(x, t, p)=\inf_{(\a, \tau, p'\in\cI_p)}\Bigg(\int_t^{\tau}e^{-\l(s-t)}\ell(y(s), \a(s), p, s)ds\\
+e^{-\l(\tau-t)}\Big(C(y(\tau), p, p')+V_{p'}(y(\tau), \tau)\Big)\Bigg).
\end{multline*}
\end{itemize}
As a consequence, $V(x,t,p)=V_p(x,t)$ for all $(x,t,p)$.
\end{proposition}
\par\smallskip
Now, we are able to obtain a differential characterization of the value function $V_p$ as viscosity solution of an Hamilton-Jacobi-Bellman variational inequality.
\par
In the sequel, by $(\cdot)_t$ and $D_x$ we denote the time derivative and the spatial gradient.
\begin{theorem}
Under the hypotheses of Proposition \ref{dynprog}, for every $p\in\cI$ the value function $V_p$ is the unique bounded and uniformly continuous viscosity solution $U$ of
\begin{equation}
\label{V}
\begin{cases}
\max\{U(x, t)-\psi_p(x, t),-U_t(x, t)+\l U(x, t)+H^p(x, t, D_xU(x, t))\}=0,\\ \mbox{\hspace{8.2cm}$(x, t)\in\mathbb{R}^d\times[0, T[$}\\
U(x, T)=\psi_p(x, T), \mbox{\hspace{6.8cm}$x\in\mathbb{R}^d$},
\end{cases}
\end{equation}
where
\begin{equation}\label{ham}
H^p(x, t, \xi)=\sup_{a\in A}\{-f(x, a, p)\cdot \xi -\ell(x, a, p, t)\}
\end{equation}
is the $p$-labeled Hamiltonian and 
\begin{equation}\label{swi}
\psi_p(x, t):=\inf_{p'\in\cI_p}(C(x, p, p')+V_{p'}(x, t))
\end{equation}
is the $p$-labeled switching operator. Moreover, the family of functions $\{V_p:p\in\cI\}$ is the unique family of bounded and uniformly continuous functions $\{U_p:p\in\cI\}$ that solves the problem
$$
\left\{
\begin{array}{ll}
\displaystyle
\text{for any $p\in\cI$, $U_p$ is the unique viscosity solution of \eqref{V}} \\
\displaystyle
\text{with }\psi_p\ \mbox{replaced by }\psi_p^U(x, t):=\inf_{p'\in\cI}(C(x, p, p')+U_{p'}(x, t)),\\
\displaystyle
U_{\bar p}=0
\end{array}.
\right.
$$
\end{theorem}

\section{Numerical testing} \label{s-test1}

In this section, we briefly describe a technique of approximation for the equation  \eqref{V}. This will be useful to produce some tests to illustrate the framework and show in practice the advantages of the method. Moreover, the same approximation scheme will be used in the mean-field  game framework.

\subsection{A numerical approximation of a variational Hamilton-Jacobi inequality}\label{s:scheme1}

We describe a numerical approximation for \eqref{V}. Let us consider a discrete grid of nodes $(x_{i},p)$ in the state space that, for simplicity, we consider as $[-c,c]^2\times \cI$, with $c\in\R_+$. Thus, a general $i=(i_1,i_2)$, $x_{i}=(-c+i_1 \Delta x , -c+i_2\Delta x)$. We proceed to define also a discrete time space as $t_k:=k\Delta t\in[0,T]$.
\par
Following \cite{MR3328207}, we write an approximation scheme for  \eqref{V} at $(x_{i},p)$ in a backward in time explicit form as
\begin{equation}\label{Scheme1}
V(x_{i},t_{k-1},p) = \min \left( NV(x_{i},t_{k},p), \Sigma\left(x_{i},t_k,p,V\right)\right).
\end{equation}
In \eqref{Scheme1}, the numerical operator $\Sigma$ is related to the continuous control, or, in other terms, to the approximation of the Hamiltonian function \eqref{ham}. 
\par
We approximate the Hamiltonian part using a semi-Lagrangian approach (see, i.e., \cite{camilli1995approximation}). The main advantage of such an approach is that we build a monotone scheme unconditional stable to the grid parameters justifying the choice of an explicit-in-time scheme. A standard semi-Lagrangian discretization of the Hamiltonian is given by
\begin{multline}\label{Scheme2}
\Sigma\left(x_{i},t_k,p,V\right)=  \lambda\Delta t V(x_{i},t_k,p)+\min_{a\in A}\left\{\mathbb{I}\left[V\right] (x_{i}-\Delta t \> f(x_{i},a,p), t_{k},p)\right.\\\left.+\Delta t \,\ell(x_{i},a,p,t_k)  \right\},
\end{multline}
where we denote by $\mathbb{I}\left[V^\Delta\right] (x,t,p)$ the values of the discrete function $V^\Delta$ computed at $(x,t,p)$ obtained using a standard interpolation operator. If the interpolation $\mathbb I$ is monotone, the resulting scheme is consistent, monotone and $L^\infty$ stable, and therefore convergent via Barles-Souganidis theorem (see \cite{BS91}). Some typical examples of monotone interpolation operators are $\PP_1$ (piecewise linear on triangles/tetrahedra) and $\Q_1$ (piecewise multilinear on rectangles) interpolations.
\par\smallskip
The discrete switch operator $N$ is computed at a node $(x_{i},p)$ as
\begin{equation}\label{Scheme11}
NV(x_{i},t_k,p) := \min_{p'\in \bar \I^p}\left\{ C(x_{i},p, p')+V(x_{i},t_k,p') \right\},
\end{equation}
which is clearly, the same operator introduced before in \eqref{swi} with the only difference that, being the operator backward in time, we look in the complementary of the admissible switches $\bar \I^p=\I\setminus \I^p$.

\subsection{An optimal visiting test for a single agent}\label{s:test1}
We illustrate the technique on a simple scenario.
We consider a classic quadratic penalization of the control in the running cost and some isotropic dynamics
$$ \ell(x,a,p,t)=\frac{\|a\|^2}{2}, \quad f(x,a,p)=a,$$
The control is chosen in the whole $A=\R^2$ (despite the fact that for classic results it is possible to prove that the optimal control is always included in a compact subset of $\R^2$, cf. \cite{BCD97}) and the problem lives in the domain $\Omega=[-1,1]^2$. The target set is composed by the points 
$$ \cT_j=0.6\,\left(\cos\left(j\frac{2\pi}{3}\right),\sin\left(j\frac{2\pi}{3}\right)\right), \quad j=1,\ldots,3.$$
The switching cost is set to 
$$ C(x,p,p')=\sum_{j\in\cI}\chi_j(p,p')\|x-\cT_j\|,$$
where $\chi_j(p,p')$ is an indicator of the possible switch relative to the $j$-target (if the player renounces to visit the $j$-target, it pays the distance from it). We underline that multiple switches are allowed by this framework:
$$
\chi_j(p,p')=\begin{cases}
0&\text{if }p_j=p'_j\\
1&\text{otherwise}
\end{cases}.
$$
We implemented the scheme \eqref{Scheme2}, where the optimization is performed by steepest descent on the control set $A$. The final time is $T=5$ while the boundary condition penalizes the states different from the final one $\bar p$, i.e. 
$$ \psi_p(x,T)=\sum_{j\in\cI}\chi_j(p,\bar p)\|x-\cT_j\|\quad \hbox{ for any }p\in \cI, \ x\in [-1,1]^2.$$
The latter means that if a trajectory ends not in the final state $\bar p$ where all targets have been considered, it pays the cost of switching to $\bar p$. This choice is made to preserve continuity on the boundary of the value function for each state.

\begin{figure}[t]
\begin{center}
\begin{tabular}{c}
\includegraphics[width=0.8\textwidth]{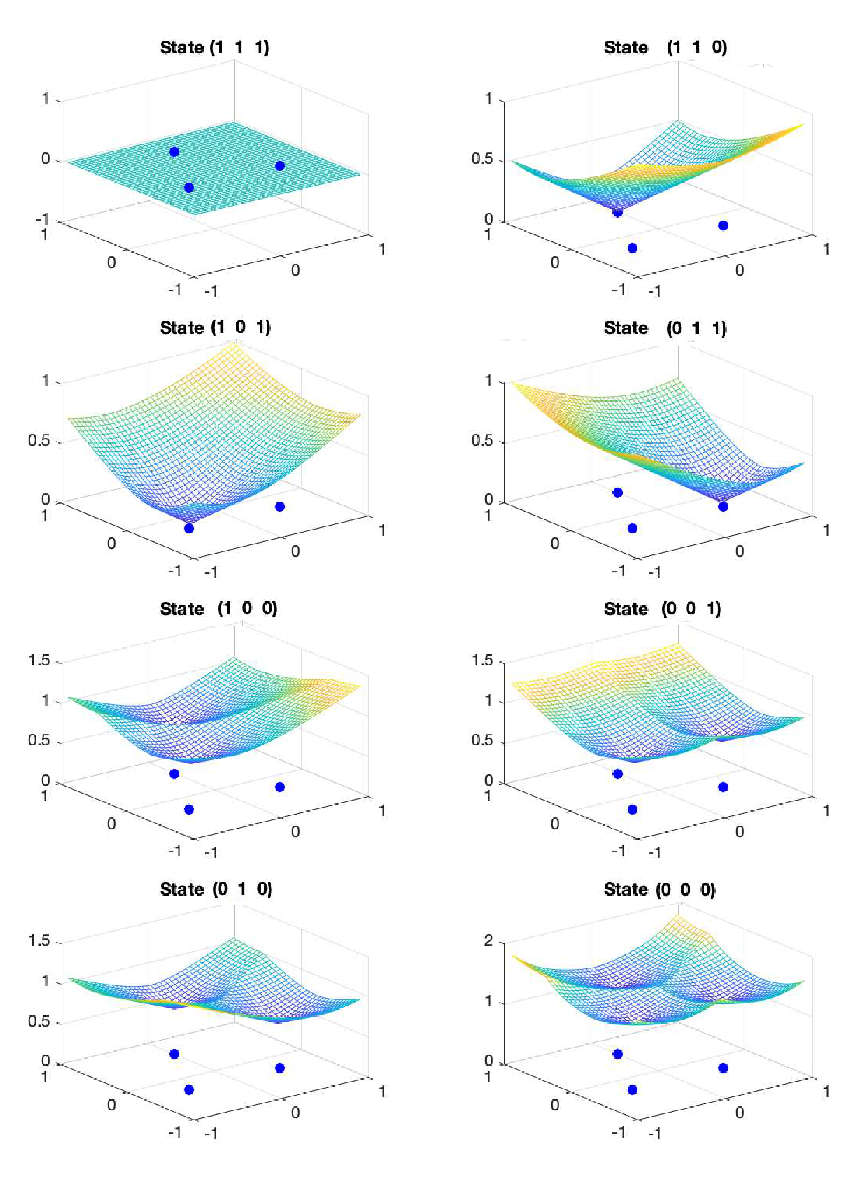}
\end{tabular}
\end{center}
\caption{Test 1. Approximated value functions in the various discrete states of the system}\label{3}
\end{figure}
\smallskip
We notice that the optimal control model that we have built tries to make any trajectory as closely as possible to all, or to the higher number of unvisited target points on the available time  $T-t$, minimizing the effort $\|\alpha\|$.
\par
Figure \ref{3} shows a collection of approximated value functions at time $T=0$ for any state of the system. The active targets (the ones still labeled as $1$ in $p$) act as attractors for the trajectory (local minimum in the value function) while they do not are highlighted if not active. 

\begin{figure}[t]
\begin{center}
\begin{tabular}{c}

\includegraphics[width=.4\textwidth]{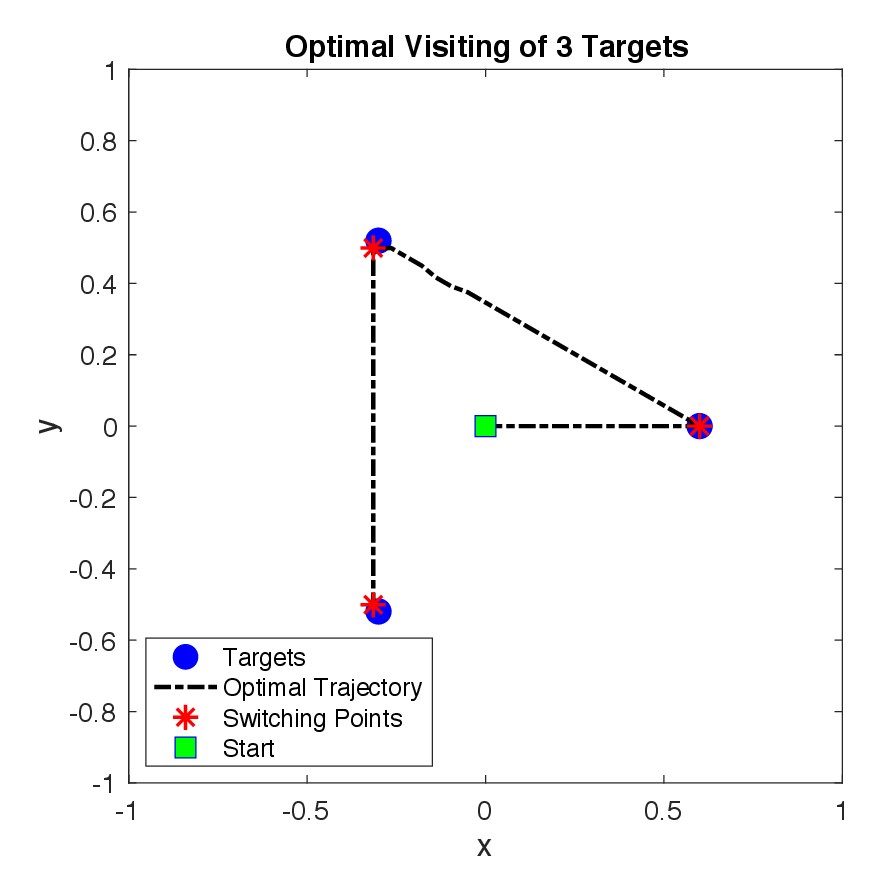}
\includegraphics[width=.4\textwidth]{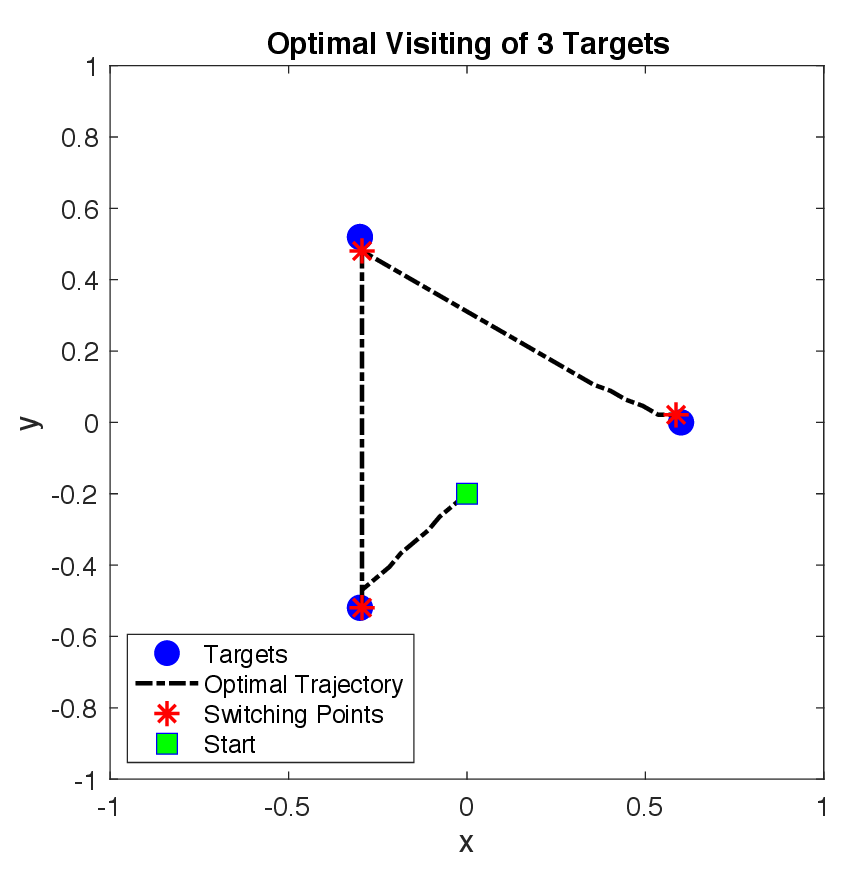}\\
\includegraphics[width=.4\textwidth]{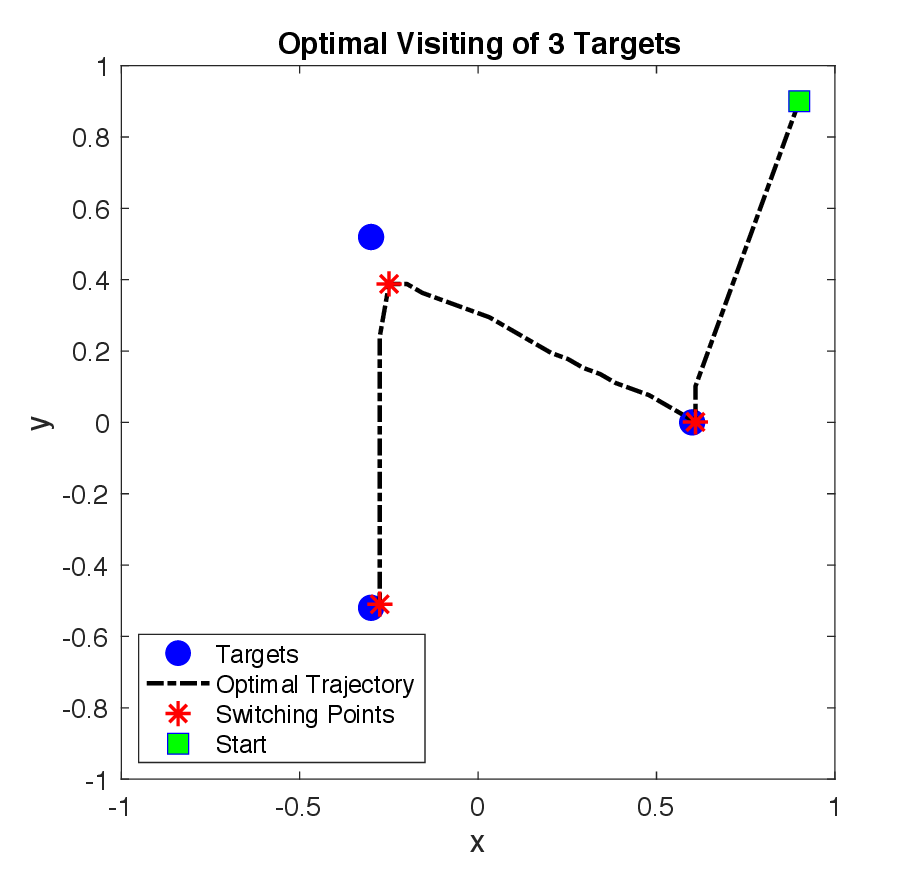}
\end{tabular}
\end{center}
\caption{Test 1. Optimal trajectories for various starting points: (top/left) $x_0=(0,0)$ (top/right) $x_0=(0,-0.2)$, (bottom) $x_0=(0.9,0.9)$. }\label{fig3}
\end{figure}
During the computation of the value function, we found the optimal control and the switching maps as arguments of the minimization \eqref{Scheme2}, \eqref{Scheme11}. We call them respectively $\alpha^*(x,t,p)$ and $\sigma^*(x,t,p)$. We build an approximation of an optimal trajectory starting from a point $x_0\in[-1,1]^2$ via the standard Euler approximation
\begin{equation*} 
\begin{cases}
Y_{x_0}(t^{n+1}):=
Y_{x_0}(t^{n})+\Delta t \alpha^*(Y_{x_0}(t^{n}),t^n,p^n)\\
p^{n+1}:=\begin{cases}
p^* & \hbox{ if }\sigma^*(Y_{x_0}(t^{n}),t^n,p^n)=p^*\\
p^{n} & \hbox{ otherwise}
\end{cases}
\\
Y_{x_0}(0)=x_0\\
p^0=\hat p
\end{cases}.
\end{equation*}
In Figure \ref{fig3}, we show some of these optimal trajectories for various choices of the starting points. We observe as, when possible, the path visits all the targets with an order that is determined by the starting point. When this is not possible (as in the case of $x_0=(0.9,0.9)$) the trajectory gets as close as possible to each target before renouncing and aiming at the next one. Outside the switching point, straight lines constitute the trajectories since the running cost is spatially and temporally homogeneous.

\section{Optimal visiting for a crowd: the continuity equation}\label{s-crowd}

In a possible study of a mean-field game for a population of agents of density $\mu$, each one of them playing a $p$-labeled optimal stopping problem like the one in Section \ref{s-single}, one would be led to consider the coupling of the system (\ref{V}) of Hamilton-Jacobi quasi-variational inequalities (coupled by the stopping costs) with a system of continuity equations (one per each level $p$ and coupled by a transfer through some sinks and sources). In particular, the sink at the level $p$ is the region where the agents stop running at level $p$ and pass to a new subsequent level $p'\in\cI_p$, and similarly for the sources. Such a coupling should provide the optimal vector field $b_p(x,t)$, giving the optimal flow, and the optimal switching time-dependent sets $\cS^t_p$ for the evolution of the masses of the agent $\mu_p$, labeled by $p$. The vector field $b_p$ and the switching sets $\cS^t_p$ will depend on the value function $V_p(x,t)$, in particular $b_p$ is typically $-D_xV_p$, see also Remark \ref{MFGobs}. An illustrative scheme of the motion rules of $\mu_p$ is sketched in Figure \ref{switch2}.

\begin{figure}[t]
\centering
\resizebox{!}{0.5\textwidth}{
\begin{tikzpicture}
\draw (0,0)--(7*0.996,7*-0.087)--(7*0.996+2*0.707, 7*-0.087+2*0.707)--(2*0.707, 2*0.707)--cycle;
\draw (0,0+2)--(7*0.996,7*-0.087+2)--(7*0.996+2*0.707, 7*-0.087+2*0.707+2)--(2*0.707, 2*0.707+2)--cycle;
\draw (0,0+4)--(7*0.996,7*-0.087+4)--(7*0.996+2*0.707, 7*-0.087+2*0.707+4)--(2*0.707, 2*0.707+4)--cycle;

\node at (7*0.996+0.1*0.707+1.5, 7*-0.087+0.9*0.707){$\bar p$};
\node at (7*0.996+0.1*0.707+1.5, 7*-0.087+0.9*0.707+2){$p''\in\cI_{p'}$};
\node at (7*0.996+0.1*0.707+1.5, 7*-0.087+0.9*0.707+4){$p'\in\cI$};

\draw[-stealth,shift={(1.5 cm, 0.4*0.707 cm)},rotate=-10] (0,0)--(0.3,0);
\draw[-stealth,shift={(2.5 cm, 1.4*0.707 cm)},rotate=-10] (0,0)--(0.3,0);
\draw[-stealth,shift={(5.2cm, 1*0.707 cm)},rotate=-10] (0,0)--(0.3,0);
\draw[-stealth,shift={(4.5 cm, -0.1*0.707 cm)},rotate=-10] (0,0)--(0.3,0);
\draw[-stealth,shift={(6 cm, -0.2*0.707 cm)},rotate=-10] (0,0)--(0.3,0);
\draw[-stealth,shift={(6.5  cm, 1*0.707 cm)},rotate=-10] (0,0)--(0.3,0);

\draw[-stealth,shift={(6.2 cm, 0.55*0.707 cm)},rotate=-25] (0,0)--(0.4,0);
\draw[-stealth,shift={(3.7 cm, 1*0.707 cm)},rotate=-25] (0,0)--(0.4,0);
\draw[-stealth,shift={(3.5 cm, 0.2*0.707 cm)},rotate=-25] (0,0)--(0.4,0);

\draw[-stealth,shift={(2.5 cm, 0.4*0.707 cm)},rotate=-10] (0,2)--(0.3,2);
\draw[-stealth,shift={(1.2 cm, 1.5*0.707 cm)},rotate=-10] (0,2)--(0.3,2);
\draw[-stealth,shift={(4cm, 1.35*0.707 cm)},rotate=-10] (0,2)--(0.3,2);
\draw[-stealth,shift={(6.5  cm, 1*0.707 cm)},rotate=-10] (0,2)--(0.3,2);

\draw[-stealth,shift={(6 cm, 0.8*0.707 cm)},rotate=-25] (0,2)--(0.4,2);
\draw[-stealth,shift={(2 cm, 1.8*0.707 cm)},rotate=-25] (0,2)--(0.4,2);

\draw[-stealth,shift={(1.7 cm, 1.7*0.707 cm)},rotate=-10] (0,4)--(0.3,4);
\draw[-stealth,shift={(1.5 cm, 0.2*0.707 cm)},rotate=-10] (0,4)--(0.3,4);
\draw[-stealth,shift={(0.7 cm, 1.5*0.707 cm)},rotate=-10] (0,4)--(0.3,4);
\draw[-stealth,shift={(3cm, 1.5*0.707 cm)},rotate=-10] (0,4)--(0.3,4);
\draw[-stealth,shift={(4 cm, -0.1*0.707 cm)},rotate=-10] (0,4)--(0.3,4);

\draw[-stealth,shift={(2 cm, 0.8*0.707 cm)},rotate=-25] (0,4)--(0.4,4);

\draw plot [smooth cycle] coordinates {(7*0.996-0.4 *7*0.996 ,2-0.1*2*0.707) (7*0.996-0.4 *7*0.996-0.3,2+0.2*2*0.707) (7*0.996-0.25*7*0.996,2+0.6*2*0.707) (7*0.996-0.1 *7*0.996,2+0*2*0.707)};
\draw plot [smooth cycle] coordinates {(7*0.996-0.4 *7*0.996+1,2-0.1*2*0.707+2) (7*0.996-0.35 *7*0.996-0.3+1,2+0.2*2*0.707+2) (7*0.996-0.2*7*0.996+1,2+0.6*2*0.707+2) (7*0.996-0.1 *7*0.996+1,2+0*2*0.707+2)};
\draw plot [smooth cycle] coordinates {(7*0.996-0.4 *7*0.996-3,2-0.1*2*0.707+2.3) (7*0.996-0.35 *7*0.996-0.3-3,2+0.2*2*0.707+2.2) (7*0.996-0.25*7*0.996-3,2+0.6*2*0.707+2) (7*0.996-0.1 *7*0.996-3,2+0*2*0.707+2.3)};

\node at (7*0.996-0.25 *7*0.996 ,2+0.2*2*0.707){$\cS^t_{p''}$};
\node at (7*0.996-0.1 *7*0.996 ,2+2+0.2*2*0.707){$\cS^t_{p'}$};
\node at (7*0.996-0.68 *7*0.996 ,2+2+0.38*2*0.707){$\cS^t_{p'}$};
%

\draw[-stealth] (7*0.996-0.1 *7*0.996 ,2+2+0.1*1*0.500)to[in=90,out=200](5.4*0.996-0.25 *7*0.996 ,2+0.25*2*0.707);
\draw[-stealth]  (7*0.996-0.25 *7*0.996 ,2+0*2*0.707) to[in=150,out=200](7*0.996-0.25 *7*0.996 ,0.25*2*0.707);
\draw[-stealth]  (7*0.996-0.8 *7*0.996 ,2+2+0.2*2*0.707)to[in=120,out=200] (7*0.996-0.75 *7*0.996 ,0.6*2*0.707);

\node at (7*0.996-0.4 *7*0.996 ,2+2+0.25*2*0.707){\footnotesize $-D_xV_{p'}(x,t) $};
\node at (7*0.996-0.68 *7*0.996 ,2+0.38*2*0.707){\footnotesize $-D_xV_{p''}(x,t) $};
\node at (7*0.996-0.7 *7*0.996 ,0.4*2*0.707){\footnotesize $-D_xV_{\bar p}(x,t) $};
\end{tikzpicture}}
\caption{The evolution of the densities $\mu_p$: when outside the switching sets, for a label $p$, the density moves accordingly to the direction $-D_xV_p$; when inside a switching set, the new label is also detected by any optimization criterium which may also depend on space and time.}\label{switch2}
\end{figure}
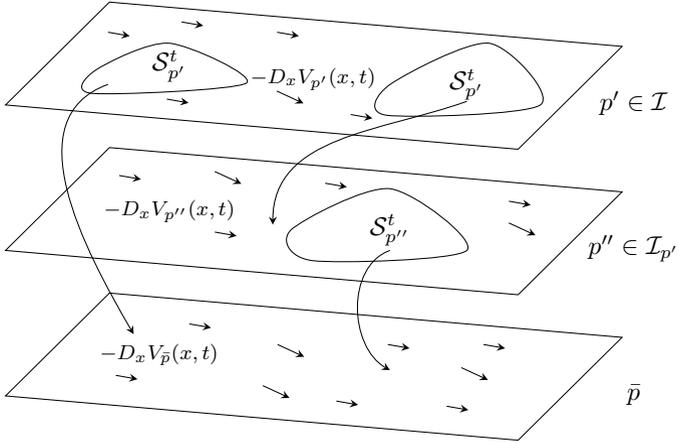

\smallskip

\par\smallskip
However, as said in the Introduction, here, from an analytical point of view, we focus only on a single continuity equation for a given suitably regular field, with possible sinks and sources and we left further analysis to future studies. In Section \ref{s-test2}, some numerical tests are shown for more general situations.

\subsection{The continuity equation with a sink}
We want to model the evolution on $\R^d$ of a mass $\mu$ subject to a given flow with a presence of a given region of $\R^d$ acting as sink: the portion of mass that possibly enters the sink instantaneously disappears. The next section will give some numerical experiments. The region representing the sink can be in general moving in time but, here, we regard it as constant. The generalization to the moving case works with same ideas and calculations. 
\par\smallskip
For the notation and the construction of the following setting, we mostly rely on \cite{notecard}. We consider a flow $\Phi:\R^d\times\R\times\R\longrightarrow\R^d$ given by the solutions of the ordinary differential system for $t\in\R$ and $x\in\R^d$,
\begin{equation}
\label{eq:system_b}
\begin{cases}
y'(s)=b(y(s), s),&s>t\\
y(t)=x
\end{cases},
\end{equation}
that is $\Phi(x,t,s)=y(s)$ solving (\ref{eq:system_b}). We will be mostly concerned with $\Phi(\cdot,0,\cdot)$. In (\ref{eq:system_b}), the field 
$b:\R^d\times\R\lra\R^d$ is assumed to be bounded, continuous and Lipschitz continuous w.r.t. $x\in\R^d$ uniformly w.r.t. $t\in\R$. Then, the flow $\Phi(\cdot, 0, \cdot)$ is Lipschitz continuous.
\par\smallskip
The sink is represented by a subset $\cS\subset\R^d$, which is assumed to be closed with compact and $C^1$ boundary. Then, for a point $x\in\R^d$, we define the possible first arrival time (to the sink) as
\begin{equation}
\label{firstat}
t_x:=\inf\{t\geq0:\Phi(x,0,t)\in\cS\},
\end{equation}
and the set of possible arrival points on the sink, for a given $t$, as
\begin{equation}
\label{switchmap}
\cS^t:=\{z\in\partial\cS:\exists x\in\R^d\ \text{such that}\ t_x=t \ \text{and}\ \Phi(x, 0, t)=z\}.
\end{equation}
We will see in \S\ref{achievetimef} that it is possible to characterize the possible first arrival time $t_x$ as the unique (viscosity) solution of an Hamilton-Jacobi-Bellman equation with suitable boundary conditions. 
\par\smallskip
We work in the set $\cG$ of positive Radon measures $\mu$ on $\R^d$ with finite first order moment, bounded by a constant $G$ (i.e., $\int_{\R^d}d\mu\leq G$ for every $\mu\in\cG$). Such a space can be endowed with the generalized Wasserstein distance (see \cite{piccross1, piccross2})
\begin{equation}
\label{wasserstein}
\cW(\mu, \mu')=\sup\left\{\int_{\R^d}\phi d(\mu-\mu'):\phi\in C^0_c(\R^d), \ \|\phi\|_{\infty}\leq1, \ \Lip(\phi)\leq1\right\}.
\end{equation}
Let
$$
\tilde m_0(x):=\begin{cases}m_0(x),&x\in\R^d\setminus\cS\\
0,&\text{otherwise}
\end{cases},\qquad\tilde m_0\in\cG,
$$
where $m_0\in\cG$ is given, be the initial distribution on $\R^d$. We suppose that it is absolutely continuous with a density, still denoted $\tilde m_0$, which is bounded and has a compact support.
\par\smallskip
Then, the continuity equation with a sink (to be interpreted in a suitable weak formulation), with a finite horizon $T>0$, is
\begin{equation}
\label{eq1}
\begin{cases}
\mu_t(x, t)+\operatorname{div}(\mu(x, t)b(x, t))+\mathbbm{1}_{\{(\cS^t, t):t\in[0, T]\}}\mu(x, t)=0,\\ \mbox{\hspace{7.99cm}$(x, t)\in\mathbb{R}^d\times[0, T]$}\\
\mu(x, 0)=\tilde m_0(x),\mbox{\hspace{7.17cm}$x\in\mathbb{R}^d$}
\end{cases}.
\end{equation}
\begin{remark} \label{MFGobs}
In the possible mean field game problem, using the notation of previous sections, at every level $p$ the sink would be given by the evolutive stopping set $\cS_p(t)=\{x\in\R^d:V_p(x, t)=\psi_p(x, t)\}$ and the field $b$ by the gradient of the value function $V_p$. Hence, the regularity assumptions above should be probably adjusted. In particular, the presence of more than one target leads to possible multiplicity of the optimal control, which makes the population split into several fractions, each one of them following one of the optimal behaviors. A similar situation is studied in \cite{bagfagmagpes}, \cite{bertucci2020}. Anyway, we may expect that the value function $V_p$ will be suitably regular along the optimal trajectories.
\par
In view of the mean field case, in \S\ref{measuredependence} we will study a possible dependence of the field $b$ on the measure and in \S\ref{s-test2} some numerical tests will be performed including this possibility. 
\end{remark}
In the sequel, we denote by $\Psi$ the inverse of the flow $\Phi$ starting from $\partial\cS$, i.e., all the states backwardly reached by the trajectories starting from the points of $\partial\cS$ in the time interval $[0, T]$. That is
$$
(z, \tau)\longmapsto\Psi(z, \tau, \tau),\quad 0\leq\tau\leq T, \ \ z\in\partial\cS
$$
with $\Psi(z, \tau, \tau)=\zeta(\tau)$ satisfying
\begin{equation}
\label{sistemadinamico}
\begin{cases}
\zeta'(s)=-b(\zeta(s), \tau-s),&0<s< \tau\\
\zeta(0)=z
\end{cases}.
\end{equation}
By hypotheses, $\Psi$ is Lipschitz continuous as $\Phi$ and it is such that
\begin{equation}
\label{psiphi}
\begin{aligned}
\Phi(\Psi(z, \tau, \tau), 0, \tau)&=z,\\
\Psi(\Phi(x, 0, \tau), \tau, \tau)&=x.
\end{aligned}
\end{equation}
Now, fixed $s\in[0, T]$, we define the sink-reaching-points set, in time $s$, as 
$$
\cB(s):=\{x\in\R^d: t_x\leq s\},
$$
that is the set of all initial points $x\in\R^d$ from which the agents enter the sink before $s$. Observe that
\begin{multline*}
\cB(s)=\bigcup_{\tau\in[0, s]}\underbrace{\{x\in\R^d:t_x=\tau\}}_{=:\cB^{\tau}}=\bigcup_{\tau\in[0, s]}\{x=\Psi(z, \tau, \tau):z\in\cS^{\tau}\}
\end{multline*}
and that
\begin{equation}
\label{inscatolamento}
s_1\leq s_2\ \Rightarrow\ \cB(s_1)\subseteq\cB(s_2).
\end{equation}
\begin{definition}
\label{defweak}
We say that $\mu$ is a {\it weak solution} to \eqref{eq1} if $\mu\in L^1([0, T], \cG)$ is such that, for any test function $\phi\in C^{\infty}_c(\R^d\times[0, T[)$, we have
\begin{multline*}
\int_{\R^d}\phi(x, 0)d\tilde m_0(x)\\+\int_0^T\int_{\R^d\setminus\Phi(\cB(t), 0, t)}(\phi_t(x, t)+\langle D_x\phi(x, t), b(x, t)\rangle)d\mu(t)(x)dt\\
-\int_0^T\int_{\R^d}\mathbbm{1}_{\{(\cS^t, t):t\in[0, T]\}}\phi(x, t)d\mu^t(t)(x)dt=0,
\end{multline*}
where $\mu^t$ entering the last integral is $\mu^t(t)=g(t)\mu^t(0)$. The measure $\mu^t(0)$ is the disintegration of $\mu(0)$ on the fibers $\cB^{\tau}$ that compose $\cB(t)$ and $g(\cdot)$ is the density of the measure $\nu$ on the indices $\tau$ of the fibers $\cB^{\tau}$ such that
$$
E\subset\cB(t)\ \Rightarrow\ \mu(0)(E)=\int_0^t\mu^{\tau}(0)(\cB^{\tau}\cap E)d\nu=\int_0^tg(\tau)\mu^{\tau}(0)(\cB^{\tau}\cap E)d\tau.
$$
\end{definition}
\par\smallskip
For $s\in[0, T]$, we define the following measure on $\R^d$
\begin{equation}
\label{muoptstop}
\tilde\mu(s)=\begin{cases}
\Phi(\cdot, 0, s)\#\tilde m_0&\text{on}\ \R^d\setminus\Phi(\cB(s), 0, s)\\
0,&\text{otherwise}
\end{cases}.
\end{equation}
Observe that the set $\Phi(\cB(s), 0, s)$ takes into account all the agents who passed through the sink $\cS$ at least once in the time interval $[0, T]$ and then disappeared. For simplicity, we set $\mu(s):=\Phi(\cdot, 0, s)\#\tilde m_0$. 
\par
By the hypotheses on $b$, $\tilde m_0$ and by \eqref{muoptstop}, we have that $\tilde\mu(s)$ is a positive Radon measure on $\R^d$ with finite first order moment. Moreover, it certainly satisfies the constraint $\int_{\R^d}d\tilde\mu(s)\leq G$ because, with respect to $\mu(s)$, it may only lose mass through the sink. Hence $\tilde\mu(s)\in\cG$. Furthermore, it is absolutely continuous with a density which is bounded and has a compact support.
\begin{lemma}
\label{lemmalip}
For $s\in[0, T]$, consider the function
\begin{equation}
\label{mappapi}
\pi:\cB(s)\longrightarrow[0, s],\qquad \pi(x):=t_x
\end{equation}
and suppose that it is Lipschitz continuous (see Remark \ref{remarksupi}). Then, the map $s\longmapsto\tilde\mu(s)$ is Lipschitz continuous in $\cG$ (with respect to the metrics \eqref{wasserstein}). 
\end{lemma}
\begin{proof}
Let $s,s'\in[0, T]$, $s\geq s'$. Then, recalling \eqref{inscatolamento}, we have
\begin{multline*}
\cW(\tilde\mu(s'), \tilde\mu(s))=\sup_{\substack{\|\phi\|_{\infty}\leq 1\\ \Lip(\phi)\leq 1}}\Bigg\{\int_{\R^d}\phi(x)d\tilde\mu(s')(x)-\int_{\R^d}\phi(x)d\tilde\mu(s)(x)\Bigg\}
\\
=\sup_{\substack{\|\phi\|_{\infty}\leq 1\\ \Lip(\phi)\leq 1}}\Bigg\{\int_{\R^d\setminus\Phi(\cB(s'), 0, s')}\phi(x)d\mu(s')(x)-\int_{\R^d\setminus\Phi(\cB(s), 0, s)}\phi(x)d\mu(s)(x)\Bigg\}\\
=\sup_{\substack{\|\phi\|_{\infty}\leq 1\\ \Lip(\phi)\leq 1}}\Bigg\{\int_{\R^d}\phi(x)d(\mu(s')-\mu(s))(x)
+\int_{\Phi(\cB(s), 0, s)\setminus\Phi(\cB(s'), 0, s)}\phi(x)d\mu(s)(x)\\+\int_{\Phi(\cB(s'), 0, s)}\phi(x)d\mu(s)(x)-\int_{\Phi(\cB(s'), 0, s')}\phi(x)d\mu(s')(x)\Bigg\}
\\
=\sup_{\substack{\|\phi\|_{\infty}\leq 1\\ \Lip(\phi)\leq 1}}\Bigg\{\int_{\R^d}(\phi(\Phi(x, 0, s'))-\phi(\Phi(x, 0, s)))d\tilde m_0+\\
\int_{\cB(s)\setminus\cB(s')}\phi(\Phi(x, 0, s))d\tilde m_0(x)
+\int_{\cB(s')}(\phi(\Phi(x, 0, s))-\phi(\Phi(x, 0, s')))d\tilde m_0(x)\Bigg\}\\
\leq\sup_{\substack{\|\phi\|_{\infty}\leq 1\\ \Lip(\phi)\leq 1}}\Bigg\{\int_{\R^d}\|\Phi(x, 0, s')-\Phi(x, 0, s)\|d\tilde m_0
+\int_{\cB(s)\setminus\cB(s')}\phi(\Phi(x, 0, s))d\tilde m_0(x)\\
+\int_{\cB(s')}\|\Phi(x, 0, s)-\Phi(x, 0, s')\|d\tilde m_0(x)\Bigg\}
\\
\leq 2GM|s'-s|+\|\tilde m_0\|_{\infty}\cL^d(\cB(s)\setminus\cB(s')),
\end{multline*}
where $M$ is the time-Lipschitz constant for $\Phi$ (i.e., the bound for $b$). Hence we conclude if we estimate $\cL^d(\cB(s)\setminus\cB(s'))$. In particular, if we prove that the map $s\longmapsto\cL^d(\cB(s))$ is Lipschitz continuous, we are done. Observe that the function
$$
f:\sigma\longmapsto\pi(\Phi(y, 0, \sigma))=\pi(y)-\sigma
$$
is such that
$$
1=|f'(0)|=|\nabla\pi(y)\cdot b(y, 0)|\leq\|\nabla\pi\|\|b\|_{\infty}\leq M\|\nabla\pi\|\quad\text{a.e.}
$$
and then $\|\nabla\pi\|\geq\frac{1}{M}>0$ almost everywhere. Therefore
\begin{multline}
\label{lipcontL}
\frac{\cL^d(\cB(s)\setminus\cB(s'))}{M}\leq\int_{\cB(s)\setminus\cB(s')}\|\nabla\pi\|dx=\int_{s'}^s\cH^{d-1}(\pi^{-1}(\tau))d\tau\\
= \int_{s'}^s\cH^{d-1}(\Psi(\cS^{\tau}, \tau, \tau))d\tau\leq K|s-s'|,
\end{multline}
where we used the Coarea Formula (see \cite{evansgariep}) and the fact that the $(d-1)$-dimensional Hausdorff measure $\cH^{d-1}(\Psi(\cS^{\tau}, \tau, \tau))$ is bounded by a constant $K>0$ since, by hypotheses, $\Psi$ is Lipschitz continuous and $\cS^{\tau}$ is compact. Then, the map $s\longmapsto\cL^d(\cB(s))$ is Lipschitz continuous and the thesis follows. 
\end{proof}
\begin{remark}
\label{remarksupi}
In general, the map $\pi$ is not Lipschitz continuous in $\cB(s)$. But, in view of the possible mean field game model, the field $b$ will be the optimal feedback for an optimal control problem with controlled dynamics from a system $y'=\a$, and hence with total controllability. Then we expect that such a Lipschitz continuity may hold and, at the moment, it is not too restrictive to assume it. Anyway, future investigations will be made on this direction. 
\end{remark}
\begin{theorem}
The map $s\longmapsto\tilde\mu(s)$ is a weak solution of \eqref{eq1}.
\end{theorem}
\begin{proof}
Let $\phi\in\cC^{\infty}_c(\R^d\times[0, T[)$. By Lemma \ref{lemmalip}, the map
$$
s\longmapsto\int_{\R^d}\phi(x, s)d\tilde\mu(s)(x)
$$
is absolutely continuous and then we have
\begin{multline*}
\frac{d}{ds}\int_{\R^d}\phi(x, s)d\tilde\mu(s)(x)=\frac{d}{ds}\int_{\R^d\setminus\Phi(\cB(s), 0, s)}\phi(x, s)d\mu(s)(x)
\\
=\frac{d}{ds}\int_{\R^d}\phi(\Phi(x, 0, s), s)d\tilde m_0(x)-\frac{d}{ds}\int_{\cB(s)}\phi(\Phi(x, 0, s), s)d\tilde m_0(x)
\\
=\int_{\R^d}(\phi_s(\Phi(x, 0, s), s)+\langle D_x\phi(\Phi(x, 0, s), s), b(\Phi(x, 0, s), s)\rangle)d\tilde m_0(x)
\\
-\frac{d}{ds}\int_{\cB(s)}\phi(\Phi(x, 0, s), s)d\tilde m_0(x)
\\
=\int_{\R^d}(\phi_s(y, s)+\langle D_x\phi(y, s), b(y, s)\rangle)d\mu(s)(y)\\-\frac{d}{ds}\int_{\cB(s)}\phi(\Phi(x, 0, s), s)d\tilde m_0(x).
\end{multline*}
We have to compute
$$
\frac{d}{ds}\int_{\cB(s)}\phi(\Phi(x, 0, s), s)d\tilde m_0(x).
$$
\par\noindent
By the Disintegration Theorem (see Remark \ref{remabs}), we get
\begin{multline*}
\frac{d}{ds}\int_{\cB(s)}\phi(\Phi(x, 0, s), s)d\tilde m_0(x)
\\
=\frac{d}{ds}\int_0^s\int_{\{x\in\R^d:t_x=\tau\}}\phi(\Phi(x, 0, s), s)d\tilde m_0^{\tau}(x)d\nu(\tau)
\\
=\frac{d}{ds}\int_0^s\int_{\{x\in\R^d:t_x=\tau\}}\phi(\Phi(x, 0, s), s)g(\tau)d\tilde m_0^{\tau}(x)d\tau
\\
=\int_{\{x\in\R^d:t_x=s\}}\phi(\Phi(x, 0, s), s)g(s)d\tilde m_0^s(x)
\\
+\int_{\cB(s)}(\phi_s(\Phi(x, 0, s), s)+\langle D_x\phi(\Phi(x, 0, s), s), b(\Phi(x, 0, s), s)\rangle)d\tilde m_0(x).
\end{multline*}
Now, recalling that $\{x\in\R^d:t_x=s\}=\Psi(\cS^s, s, s)$ by definition, we have
\begin{multline*}
\int_{\{x\in\R^d:t_x=s\}}\phi(\Phi(x, 0, s), s)g(s)d\tilde m_0^s(x)
\\
=\int_{\Psi(\cS^s, s, s)}\phi(\Phi(x, 0, s), s)g(s)d\tilde m_0^s(x)
=\int_{\cS^s}\phi(y, s)d\mu^s(s)(y),
\end{multline*}
where $\mu^s(s):=g(s)(\Phi(\cdot, 0, s)\#\tilde m_0^s)$. Finally, we obtain
\begin{multline*}
\frac{d}{ds}\int_{\R^d}\phi(y, s)d\tilde\mu(s)(y)\\
=\int_{\R^d\setminus\Phi(\cB(s), 0, s)}(\phi_s(y, s)+\langle D_x\phi(y, s), b(y, s)\rangle)d\mu(s)(y)
\\
-\int_{\cS^s}\phi(y, s)d\mu^s(s)(y).
\end{multline*}
Since $\tilde\mu(0)=\tilde m_0$, integrating this between $0$ and $T$ we get
\begin{multline*}
\int_{\R^d}\phi(y, 0)d\tilde m_0(y)
\\
+\int_0^T\int_{\R^d\setminus\Phi(\cB(s), 0, s)}(\phi_s(y, s)+\langle D_x\phi(y, s), b(y, s)\rangle)d\mu(s)(y)ds
\\
-\int_0^T\int_{\cS^s}\phi(y, s)d\mu^s(s)ds=0,
\end{multline*}
that is
\begin{multline*}
\int_{\R^d}\phi(y, 0)d\tilde m_0(y)
\\
+\int_0^T\int_{\R^d\setminus\Phi(\cB(s), 0, s)}(\phi_s(y, s)+\langle D_x\phi(y, s), b(y, s)\rangle)d\mu(s)(y)ds
\\
-\int_0^T\int_{\R^d}\mathbbm{1}_{\{(\cS^s, s):s\in[0, T]\}}\phi(y, s)d\mu^s(s)(y)ds=0.
\end{multline*}
\end{proof}
\begin{remark}
\label{remabs}
The Disintegration Theorem in the previous proof is applied as follows: we set $Y=\cB(s)$, $X=[0, s]$ and we consider the map \eqref{mappapi}
$$
\pi:Y\longrightarrow X,\qquad \pi(x)=t_x
$$
and $\nu=\pi\#\tilde m_0\in\cG(X)$. In this way $\pi^{-1}(\tau)=\{x\in\R^d:t_x=\tau\}$ for every $\tau	\in[0, s]$. Then, there exists a $\nu$-almost everywhere uniquely determined family $\{\tilde m_0^{\tau}\}_{\tau\in[0, s]}\subset\cG(Y)$ such that for every Borel measurable function $f:Y\longrightarrow[0, +\infty]$,
$$
\int_Yf(y)d\tilde m_0(y)=\int_0^s\int_{\{x\in\R^d:t_x=\tau\}}f(y)d\tilde m_0^{\tau}(y)d\nu(\tau).
$$
Moreover, in view of \eqref{lipcontL}, that is the Lipschitz continuity of the map $s\longmapsto\cL^d(\cB(s))$, the measure $\nu$ is absolutely continuous on $X$ with a $L^{\infty}$ density denoted by $g$. 
\end{remark}
\begin{theorem}
The continuity equation \eqref{eq1} has a unique solution given by $s\longmapsto\tilde\mu(s)$.
\end{theorem}
\begin{proof}
Let $\phi\in\cC^{\infty}(\R^d)$ with $\operatorname{supp}(\phi)\subset\R^d\setminus\Phi(\cB(t), 0, t)$ for every $t\leq T$. Fix $t\leq T$ and let us consider the map
\begin{equation}
\label{mappaw}
w:\R^d\times[0, t]\longrightarrow\R,\qquad w(x, s):=\phi(\Phi(x, 0, t-s)).
\end{equation}
Then, $w$ is Lipschitz continuous in both variables $(x, s)\in\R^d\times[0, t]$ with $\operatorname{supp}(w)\subset(\R^d\setminus\Phi(\cB(s), 0, s))\times[0, t]$. Moreover, by \eqref{psiphi} we have
$$
\phi(x)=w(\Psi(x, t-s, t), s)=\phi(\Phi(\Psi(x, t-s, t), 0, t-s))
$$
and, recalling that $\Psi(x, t, t)$ is the solution of \eqref{sistemadinamico} with $\tau=t$, the function $w$ satisfies
\begin{multline*}
0=\frac{d}{ds}\phi(x)=w_s(\Psi(x, t-s, t), s)\\+\langle D_xw(\Psi(x, t-s, t), s), b(\Psi(x, t-s, t), s)\rangle\ \ \text{a.e.}
\end{multline*}
and hence, in general,
$$
w_s(y, s)+\langle D_xw(y, s), b(y, s)\rangle=0\quad\text{a.e. in}\ \R^d\times(0, t).
$$
Using $w$ as a test function for a generic $\mu$ satisfying Definition \ref{defweak}, for almost all $s\leq t$ we have
\begin{multline*}
\frac{d}{ds}\int_{\R^d}w(y, s)d\mu(s)(y)=\int_{\R^d}w_s(y, s)d\mu(s)(y)+\int_{\R^d}w(y, s)d\mu_s(s)(y)
\\
=\int_{\R^d}(-\langle D_xw(y, s), b(y, s)\rangle+\langle D_xw(y, s), b(y, s)\rangle)d\mu(s)(y)=0
\end{multline*}
since $\operatorname{supp}(w)\subset(\R^d\setminus\Phi(\cB(s), 0, s))\times[0, t]$, which implies
\begin{multline*}
\int_{\cS^s}w(y, s)d\mu^s(s)(y)\\=\int_{\Phi(\cB(s), 0, s)}(w_s(y, s)+\langle D_xw(y, s), b(y, s)\rangle)d\mu(s)(y)=0,\ \ s\leq t.
\end{multline*}
Therefore, integrating between $0$ and $t$, we get 
$$
\int_{\R^d}w(y, t)d\mu(t)(y)=\int_{\R^d}w(y, 0)d\mu(0)(y)
$$
and then
$$
\int_{\R^d}\phi(y)d\mu(t)(y)=\int_{\R^d}\phi(\Phi(y, 0, t))d\tilde m_0(y),
$$
which shows that $\mu(t)=\Phi(\cdot, 0, t)\#\tilde m_0$ on $\R^d\setminus\Phi(\cB(t), 0, t)$.
\par\smallskip
Now we have to prove that $\mu(t)=0$ on $\Phi(\cB(t), 0, t)$, that is
$$
\int_{\R^d}\phi(x)d\mu(t)(x)=0
$$
for any $\phi\in\cC^{\infty}(\R^d)$ with $\operatorname{supp}(\phi)\subset\Phi(\cB(t), 0, t)$ for every $t\leq T$.
Again fix $t\leq T$ and let us consider the map \eqref{mappaw}. Then, proceeding as before, we get
$$
w_s(y, s)+\langle D_xw(y, s), b(y, s)\rangle=0\quad\text{a.e. in}\ \R^d\times(0, t).
$$
Using $w$ as a test function for a generic $\mu$ satisfying Definition \ref{defweak}, for almost all $s\leq t$ we have
\begin{multline*}
\frac{d}{ds}\int_{\R^d}w(y, s)d\mu(s)(y)=\int_{\R^d}w_s(y, s)d\mu(s)(y)+\int_{\R^d}w(y, s)d\mu_s(s)(y)
\\
=\int_{\R^d}(-\langle D_xw(y, s), b(y, s)\rangle+\langle D_xw(y, s), b(y, s)\rangle)d\mu(s)(y)
\\
-\int_{\Phi(\cB(s), 0, s)}(w_s(y, s)+\langle D_xw(y, s), b(y, s)\rangle)d\mu(s)(y)-\int_{\cS^s}w(y, s)d\mu^s(s)(y)
\\
=-\int_{\Phi(\cB(s), 0, s)}(w_s(y, s)+\langle D_xw(y, s), b(y, s)\rangle)d\mu(s)(y)\\-\int_{\cS^s}w(y, s)d\mu^s(s)(y).
\end{multline*}
Now, observe that
\begin{multline*}
\int_{\cS^s}w(y, s)d\mu^s(s)(y)+\int_{\Phi(\cB(s), 0, s)}(w_s(y, s)+\langle D_xw(y, s), b(y, s)\rangle)d\mu(s)(y)
\\
=\int_{\{x\in\R^d:t_x=s\}}w(\Phi(y, 0, s), s)g(s)d\mu^s(0)(y)
\\
+\int_{\cB(s)}(w_s(\Phi(y, 0, s), s)+\langle D_xw(\Phi(y, 0, s), s), b(\Phi(y, 0, s), s)\rangle)d\mu(0)(y)
\\
=\frac{d}{ds}\int_{\cB(s)}w(\Phi(y, 0, s), s)d\mu(0)(y).
\end{multline*}
Then, by \eqref{mappaw} and the semigroup property of the flow $\Phi$, since $\mu(0)=\tilde m_0$ we get
$$
\frac{d}{ds}\int_{\R^d}w(y, s)\mu(s)(y)=-\frac{d}{ds}\int_{\cB(s)}\phi(\Phi(y, 0, t))d\tilde m_0(y)
$$
and hence, integrating between $0$ and $t$, 
$$
\int_{\R^d}w(y, t)d\mu(t)(y)=\int_{\R^d}w(y, 0)d\mu(0)(y)-\int_0^t\frac{d}{ds}\int_{\cB(s)}\phi(\Phi(y, 0, t))d\tilde m_0(y).
$$
Therefore
\begin{multline*}
\int_{\R^d}\phi(y)d\mu(t)(y)=\int_{\R^d}\phi(\Phi(y, 0, t))d\tilde m_0(y)-\int_{\cB(t)}\phi(\Phi(y, 0, t))d\tilde m_0(y)
\\
+\int_{\cB(0)}\phi(\Phi(y, 0, t))d\tilde m_0(y).
\end{multline*}
Since $\tilde m_0=0$ in $\cS=\cB(0)$ and $\supp(\phi)\subset\Phi(\cB(t), 0, t)$, the thesis follows.
\end{proof}

\subsection{On the measure dependence of the field $b$}
\label{measuredependence}
We consider now the field $b$ depending on the measure, that is
$$
\begin{aligned}
b: C^0([0, T], \cG)\times\R^d\times[0, T]&\lra\R^d,\\ (\mu, x, t)&\longmapsto b(\mu, x, t).
\end{aligned}
$$
We suppose that it is bounded and continuous in the whole entry $(\mu, x, t)$ and Lipschitz continuous w.r.t. $x\in\R^d$ uniformly w.r.t. $(\mu, t)\in C^0([0, T], \cG)\times[0, T]$, that is, there exists $L>0$ such that
$$
\|b(\mu, x, t)-b(\mu, y, t)\|\leq L\|x-y\|,\quad\forall x,y\in\R^d,\ (\mu, t)\in C^0([0, T], \cG)\times[0, T].
$$
In the evolution of the flow given by $b$, the sink is always represented by $\cS$. In the sequel, for every $\mu\in C^0([0, T], \cG)$ fixed, we will also use the notation
$$
\begin{aligned}
b[\mu]:\R^d\times[0, T]&\lra\R^d\\(x, t)&\longmapsto b[\mu](x, t):=b(\mu, x, t)
\end{aligned}
$$
and we consider the corresponding flow with sink evolution given by the field $b[\mu]$. As, for every fixed $\mu$, we denote by $\tilde\mu$ the unique solution of the corresponding problem \eqref{eq1}, which is, for every $t\in[0, T]$, 
\begin{equation}
\label{defmut}
\tilde\mu(t)=\begin{cases}
\Phi[\mu](\cdot, 0, t)\#\tilde m_0&\text{on }\R^d\setminus\Phi[\mu](\cB[\mu](t), 0, t)\\
0,&\text{otherwise}
\end{cases},
\end{equation}
where $\Phi[\mu]$ is the flow generated by the field $b[\mu]$ and $\cB[\mu](\cdot)$ is the corresponding sink-reaching-points set. We also denote by $\pi[\mu]$ the map as in \eqref{mappapi}. 
\begin{theorem}
Let us suppose that $\pi[\mu]$ is Lipschitz continuous uniformly in $\mu\in C^0([0, T], \cG)$ (see Remark \ref{remarksuip}). For every $\mu\in C^0([0, T], \cG)$, we have $\tilde\mu\in C^0([0, T], \cG)$. Moreover, the function
$$
\psi:C^0([0, T], \cG)\lra C^0([0, T], \cG),\qquad\psi(\mu):=\tilde\mu
$$
has a fixed point in $C^0([0, T], \cG)$. This means that the problem of flow with sink and with field depending on the measure has a solution. 
\end{theorem}
\begin{proof}
At first observe that, under the previous hypotheses, by analogous considerations as in the case with no measure dependence we have $\tilde\mu(t)\in\cG$ for all $t\in[0, T]$. 
\par\smallskip
Let us prove that $\tilde\mu\in C^0([0, T], \cG)$. In particular we have to prove that, whenever $t_n\ra t$ in $[0, T]$, then $\tilde\mu(t_n)\ra\tilde\mu(t)$ weakly-star. This comes from standard regularity results for the push-forward measures (see \eqref{defmut}) and from the fact that $\cB[\mu](t_n)\ra\cB[\mu](t)$ in the Hausdorff metrics and as $d$-dimensional Lebesgue measure, and hence similarly for $\Phi[\mu](\cB[\mu](t_n), 0, t_n)$ and $\Phi[\mu](\cB[\mu](t), 0, t)$. 
\par\smallskip
Now, to prove the second statement of the theorem, we have to show that the function $\psi$ is continuous and compact. In this way, we can conclude by the Schauder fixed point Theorem.
\par\smallskip
At first we prove the continuity of $\psi$. Let $\mu_n\ra\mu$ in $C^0([0, T], \cG)$. We have to prove that $\tilde\mu_n\ra\tilde\mu$. Let us consider the two trajectories
$$
x_n(t)=x_0+\int_0^tb[\mu_n](x_n(s), s)ds,\qquad x(s)=x_0+\int_0^tb[\mu](x(s), s)ds
$$
and we prove that $x_n$ uniformly converges to $x$ on compact sets. Indeed, by the continuity and boundedness hypotheses, the field $b$ is bounded and uniformly continuous on $(\{\mu_n\}_n\cup\{\mu\})\times K\times[0, T]$, where $\{\mu_n\}_n\cup\{\mu\}$ is the compact set in $C^0([0, T], \cG)$ given by the whole sequence with its limit, and $K\subset\R^d$ is compact. Then, the sequence $b[\mu_n]$ is bounded and equicontinuous on $K\times[0, T]$ and moreover it pointwise converges to $b[\mu]$. Hence, by the Ascoli-Arzel\`a Theorem, the convergence is uniform on $K\times[0, T]$. From this we deduce the desired uniform convergence of the trajectories. Therefore, also the flows $\Phi[\mu_n]$ uniformly converge to $\Phi[\mu]$ on compact sets. From this we have that $\cB[\mu_n](\cdot)$ converges to $\cB[\mu](\cdot)$ in the Hausdorff distance and uniformly in time, and we conclude the convergence of $\tilde\mu_n$ to $\tilde\mu$ in $C^0([0, T], \cG)$, that is, for every $\phi\in C_c^0(\R^d)$,
$$
\sup_{t\in[0, T]}\left|\int_{\R^d}\phi(x)d\mu_n(t)(x)-\int_{\R^d}\phi(x)d\mu(t)(x)\right|\ra0\quad\text{as}\ n\ra+\infty.
$$
\par\smallskip\noindent
It remains to prove that $\psi$ has compact image. We can restrict to measures on a compact set $K\subset\R^d$ independent of $\mu$, which contains all the possible compact supports of the measures $\tilde\mu$ (because of bounded dynamics and finite horizon). Then, since $s\longmapsto\tilde\mu(s)$ is Lipschitz continuous uniformly in $\mu$ (similarly as in Lemma \ref{lemmalip} and using the hypothesis on $\pi[\mu]$) with values in the compact set $\cG$ of Radon measures on $K$, we get the desired conclusion by Ascoli-Arzel\`a Theorem.
\end{proof}
\begin{remark}
\label{remarksuip}
In the previous proof, we used the fact that $s\longmapsto\tilde\mu(s)$ is Lipschitz continuous uniformly in $\mu$ and this comes from the Lipschitz continuity of $\pi[\mu]$ uniformly w.r.t. $\mu$. However, just assuming only the Lipschitz continuity of $\pi[\mu]$, and not necessarily uniformly in $\mu$, after some calculations one can prove the equicontinuity w.r.t. $\mu$ of $s\longmapsto\tilde\mu(s)$ and then still apply Ascoli-Arzel\`a Theorem.
\end{remark}

\subsection{A differential characterization of the first arrival time}
\label{achievetimef}
In this paragraph, we see a characterization of the possible first arrival time $t_x$ \eqref{firstat} to the sink $\cS$ as the unique viscosity solution of an Hamilton-Jacobi-Bellman equation with proper boundary conditions. To do this, we define
$$
\begin{aligned}
&\Tau:\R^d\times[0, T]\lra[0, +\infty[,\\ &\Tau(x, t):=t_{(x, t)}=\inf\{\zeta\geq t:\Phi(x, t, \zeta)\in\cS\}.
\end{aligned}
$$
Fixed $s\in[t, T]$, the sink-reaching-points set, that is the set from which the agents are able to reach the target $\cS$, is 
$$
\cB(s)=\{(x, t)\in\R^d\times[0, T]:\Tau(x, t)\leq s\}.
$$
Let us set 
$$
\cB:=\bigcup_{s\in[t, T]}\cB(s)=\{(x, t)\in\R^d\times[0, T]:\Tau(x, t)\leq T\}.
$$
We suppose that $\Tau$ is continuous in $\cB$ and that 
\begin{equation}
\label{Talbordo}
\Tau(x, t)=T \ \ \text{for any}\ (x, t)\in\partial\cB.
\end{equation}
As with similar calculations as in \cite{BCD97}, taking into account that now $\Tau$ is time dependent (and then the associated Hamilton-Jacobi-Bellman equation is evolutive) and that there is no control, we have the following
\begin{proposition}
The function $\Tau$ is the unique viscosity solution of 
$$
\begin{cases}
-\Tau_t(x, t)-b(x, t)\cdot D_x\Tau(x, t)-1=0,&(x, t)\in\cB\\
\Tau(x, t)=t,&(x, t)\in(\partial\cS,t)\\
\Tau(x, t)=T,&(x, t)\in\partial\cB
\end{cases}.
$$
\end{proposition}
\par\smallskip
Hence, to conclude, the possible first arrival time $t_x$ to the sink $\cS$ is given by the solution $\Tau(x, t)$ of the above problem at the time $t=0$. 
\begin{remark} Observe that the hypotheses in this section are not too restrictive to be assumed. Indeed, a priori, $\Tau$ may be discontinuous in $\cB$. But, as we noticed in Remark \ref{remarksupi} too, in view of the possible mean field game model, the field $b$ comes as the optimal feedback from a system like $y'=\a$, with total controllability. Then it is not too strong to assume that $\Tau$ is continuous in $\cB$. Moreover, again in view of the possible mean field model, condition \eqref{Talbordo} holds only on the boundary which is not viable. 
\end{remark}

\section{Approximation of the continuity equation and coupled discrete system}\label{s-test2}

We want to build an approximation scheme for the continuity equation \eqref{eq1}. We include the possibility of a non linear, non local dependency of the vector field $b(\mu,x,t):=b[\mu](x,t)$ on the density $\mu$.  In similarly  to what already done for the Hamilton-Jacobi equation, we propose a semi-Lagrangian scheme to approximate the equation. The advantages of the choice of such framework are the beneficial stability properties of the scheme without restrictions on the choice of the parameters and some monotone properties that are a direct consequence of the formulation of the scheme.  The scheme is an adaptation of the one proposed in \cite{Carlini20154269,carlini2016PREPRINT}. We also refer the reader to \cite{FestaTosinWolfram}, where a similar scheme has been applied  to a non-linear continuity equation modeling  a kinetic pedestrian model.  
\par\smallskip
The scheme is a discretization of the representation formula \eqref{muoptstop}, where the initial density $\mu_0$ is pushed by a flow $\Phi$. We assume to know the vector field $b$ and the optimal stopping set $\cS$ and we proceed by steps: first of all we restrict ourselves on a finite domain $\Omega$, and $b[\mu]: \Omega \times[0,T]\lra\R$  is the restriction of the given smooth vector field which is a data of the problem. Additionally, $\mu_0:=\tilde m_0$ is a smooth initial datum defined on $\Omega$. Formally, at time $t\in [0,T]$ the solution of  \eqref{eq1} is given, implicitly, by the image of the measure $\mu_0 d x $ induced by the flow $x\in \Omega \longmapsto \Phi(x,0,t)$, where, given $0\leq s \leq t \leq T$, $\Phi(x,s,t)$ denotes the solution of 
$$
\begin{cases}
\dot x(t)=b[\mu](x,t),&t\in ]s,T]\\
x(s)=x
\end{cases}.
$$
Given the discretization parameters $\Delta x$, $\Delta t$, we construct a grid on $\Omega\times[0,T]$ as in Section \ref{s:scheme1}. We call $M,N$ the number of elements respectively in the spatial and in the temporal grid.
\par
We discretize \eqref{eq1} using its representation formula by means of the flow $\Phi$. This can be considered an extension of the \emph{characteristic method} for these kind of equation (we take the idea of using \eqref{muoptstop} instead of a direct discretization of the equation from \cite{CDY}).
For $j\in \{1,\ldots,M\}$, $k=0,\ldots, N-1$,  we  define the discrete  characteristics (obtained by an explicit Euler scheme) as
\begin{equation*}\label{car} 
\Phi_{j,k}[\mu]:= x_{j}+\Delta t \,b[\mu](x_j,t_k).
 \end{equation*}
 where we recall $t_k:=k\Delta t\in[0,T]$.
Let us define $\{{\beta_{j}} \}$ the set of base functions, such that $\beta_j(x_i)=\delta_{i,j}$  (the Kronecker symbol) and $\sum_j \beta_j(x)=1$ for each $x\in \overline\Omega$. We approximate the solution  $\mu$ of the problem \eqref{eq1} in a finite-volumes way by a sequence $\{\Upsilon(x_i,t_k)\}$, in the space $\cG$ of the piecewise linear functions on the discrete grid, where 
$$\Upsilon(x_i,t_k) \approx\frac{1}{|E_i|}\int_{E_i} \mu(x_i,t_k) d x,$$
with, in the $d=2$ case, $E_i=[x_i-e_1\Delta x/2, x_i+e_1\Delta x/2]\times[x_i-e_2 \Delta x/2, x_i+e_2\Delta x/2]$ with $e_1=(1,0)^T$ and $e_2=(0,1)^T$ and therefore $|E_i|=\Delta x^2$. 
We compute the  sequence  $\{\Upsilon(x_i,t_k)\}$ by the following explicit scheme:  
\begin{equation}\label{schemefp}
\begin{cases}
\Upsilon(x_i,t_{k+1})= G(\Upsilon,i,k),&k=0,\ldots, N-1,\ i=1,\ldots,M,\\
\Upsilon(x_i,0)=\frac{1}{\Delta x^2} \int_{E_{i}}\mu_{0}(x) d x,&i=1,\ldots,M
\end{cases},
\end{equation}
where 
 $G:\cG\times \{1,\ldots,M\}\times\{0,\ldots,N-1\}\longrightarrow \R$ is defined by
\begin{eqnarray*}\label{definicionG}
 & G (\Upsilon,i,k) :=   \sum_{j=1}^M \mathbbm{1}_{S(x_j,t_k)}\Upsilon(x_j, t_k)
\beta_{i} \left(\Phi_{j,k}\left[ \Upsilon\right]\right),\\
& S(x_j,t_k):= 
\begin{cases}
0 & \hbox{if }x\in\cS^{t_k},\\
1 & \hbox{elsewhere}.
\end{cases}
\end{eqnarray*} 
The role of the term $G$ is to sum up the contribution of any portion of $\Upsilon$ at state $(x_j,t_k)$ such that the approximation of the characteristics starting from there falls inside $E_i$. Essentially the scheme ``pushes'' $\Upsilon$ to move in the direction given by the discrete flow $\Phi$. We observe also as the function $S(x_j,t_k)$ above  approximates the action of the term $\mathbbm{1}_{\{(\cS^t, t):t\in[0, T]\}}\mu(x, t)$ in  \eqref{eq1} and it cancels the contribution of the density in the regions where the optimal stopping occurs.

We underline, once again, that thanks to the semi-Lagrangian structure of the scheme, the convergence of it does not rely on a CFL like condition, so we are free to choose the discretization parameters $\Delta x$ and $\Delta t$  as we like. The convergence of a similar scheme has been proved in \cite{Carlini20154269} for the 1D case and in \cite{Carlini2018} in a more general setting.

We pass now to describe the discrete system obtained by a fully coupling between \eqref{V}-\eqref{eq1}. To avoid repetition, we illustrate the hybrid case where a density $\Upsilon$ can switch between various states $\cI$. This means that in the continuity equation there will be considered both pit than source points. The case of the optimal stopping, covered by the theory discussed before, is the simplified case where $\cI=\{0, 1\}$. We underline as, the existence of an equilibrium state has not analytically proven yet, the latter means that in the last test of the following section we limit ourselves to show some promising numerical evidence and to illustrate the basic features of the model that we obtain.

The coupling between \eqref{V}-\eqref{eq1} for a generic collections of discrete states $\cI$ brings us to consider the system
\begin{equation}\label{SchemeFull}
\begin{cases}
V(x_{i},t_{k-1},p) = \min \left( NV(x_{i},t_{k},p), \Sigma\left(x_{i},t_k,p,V,\Upsilon \right)\right), & \, i=1,\ldots,M,\\
\Upsilon(x_i,t_{k+1},p)= \overline G(\Upsilon,V,i,k,p),&k=\,1,\ldots, N-1,\\
V(x_{i},T,p) =\psi_p(x_i),& p\in\cI,\\
\Upsilon(x_i,0,p)=\frac{ \int_{E_{i}}\mu_{0}(x) d x}{\Delta x^2},
\end{cases}
\end{equation}
where the discrete operators $NV$ and $\Sigma$ are as defined in Section \ref{s-test1}. Additionally, 
 $\overline G$ is defined by
\begin{eqnarray*}\label{definicionbarG}
& \overline G (\Upsilon,V,i,k,p) :=   \sum_{q\in\cI}\sum_{j=1}^M \mathbbm{1}_{S(x_j,t_k,q)=p}\Upsilon(x_j, t_k,q)
\beta_{i} \left(\Phi_{j,k}\left[ \Upsilon,V\right]\right),\\
& S(x_j,t_k,q):= \arg\min_{p\in \cI} V(x_i,t_k,p).
\end{eqnarray*} 
We observe how the discrete system \eqref{SchemeFull} is composed by a forward in time $\Upsilon$ and a backward in time $V$. In the mean field game case (which is, at the moment, not fully covered by the theory developed in this paper) $\Sigma$ depends on the approximation of the density $\Upsilon$ (typically by a penalization of the running cost term) and the flux $\Phi$ aims in a direction determined by the potential $V$. This means that in the case of mutual dependency between the two variables the solution must be found as an equilibrium point between the two processes (see scheme in Figure \ref{sch}) as we shortly discuss in the last test.
\par

\begin{figure}[t]
\centering
\resizebox{!}{0.42\textwidth}{
\begin{tikzpicture}
\draw (0,0)--(0,2)--(8,2)--(8,0)--cycle;
\draw (0,4)--(0,6.3)--(8,6.3)--(8,4)--cycle;

\node at (-2,3) {$-D_xV (x,t,p), \cS^t(x, p)$};
\node at (10,3) {$\ell [Y] (x,t,p)$};

\draw (-0.3,1) to[out=180,in=270](-2,2.5);
\draw[-stealth] (-2,3.5) to[out=90,in=180](-0.3,5.15);

\draw[-stealth]  (10,2.5) to[out=270,in=0](8.3,1);
\draw (8.3,5.15) to[out=0,in=90](10,3.5);

\draw[stealth-] (1,0.2)--(7,0.2);
\draw[-stealth] (1,4.2)--(7,4.2);

\node at (4,1.5) {\textbf{Hamilton-Jacobi inequality}};
\node at (4,1) {computes $V$};
\node at (4,0.5) {backward in time};

\node at (4,5.7) {\textbf{Continuity Equation}};
\node at (4,5.2) {moves the density $Y$};
\node at (4,4.7) {forward in time};
\end{tikzpicture}}
\caption{The coupling between continuity equation and Hamilton-Jacobi variational inequality.}\label{sch}
\end{figure}
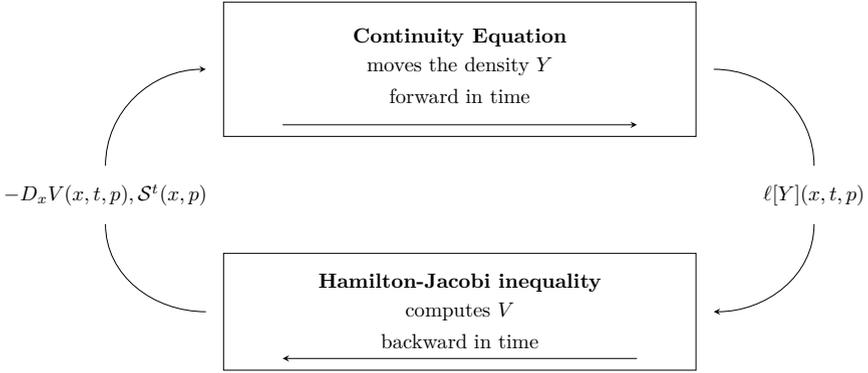

\subsection{An optimal stopping test for a crowd}

Let us consider the simple scenario of an optimal stopping problem for a density of agents. We recall that in this case, $\cI=\{0,1\}$. Therefore, the only interesting state where to compute \eqref{SchemeFull} is $p=0$.
In the domain $\Omega=[-1,1]^2$ an initial density distribution of
$$  \mu_0(x)=e^{-8\|x\|^2} \quad \hbox{ restricted to }\Omega,$$
aims to reach a single target point $\T=(0,0.6)$. We use the optimal control vector field of the control problem associated \eqref{V} as motion vector field for the continuity equation \eqref{eq1} (i.e. $b(x,t)=-D_x V_p(x,t)$) and the switching (or stopping) map $\cS^t$ as in \eqref{switchmap}.
\par\smallskip
For the other operators of the optimal stopping problem, we consider, as before, quadratic penalization of the control in the running cost and isotropic dynamics
$$ \ell(x,a,p,t)=\frac{\|a\|^2}{2}, \quad f(x,a,p)=a.$$
The control is in $A=\R^2$ and the switching cost is set as in test \ref{s:test1} to 5
$$ C(x,p,p')=\sum_{j\in\cI}\chi_j(p,p')\|x-\cT_j\|. $$
The final time is $T=0.26$ while the boundary condition at time $t=T$ is
$$ \psi_p(x,T)=2\sum_{j\in\cI}\chi_j(p,\bar p)\|x-\cT_j\|\quad \hbox{ for any }p\in \cI, \;x\in [-1,1]^2.$$
This means that at any moment an agent can choose to aim to the target lowering its distance from it (paying the relative running cost) or renounce to it paying the switching cost and exit from the game.
\par\smallskip
It is worth to underline that in this case, the coupling between equation \eqref{V} and \eqref{eq1} is only in one direction. The solution of \eqref{V} is independent from the distribution of the density therefore $b(x,t)$ is independent from $\mu(x,t)$. Hence we can proceed, first of all, to the resolution of \eqref{V} with discretization parameters $\Delta x=0.04$ and $\Delta t=0.02$, obtaining a control map $b(x,t)$ and a switching map  $\cS^t$ as reported in Figure \ref{4}.

\begin{figure}[t]
\begin{center}
\begin{tabular}{c}
\includegraphics[width=0.8\textwidth]{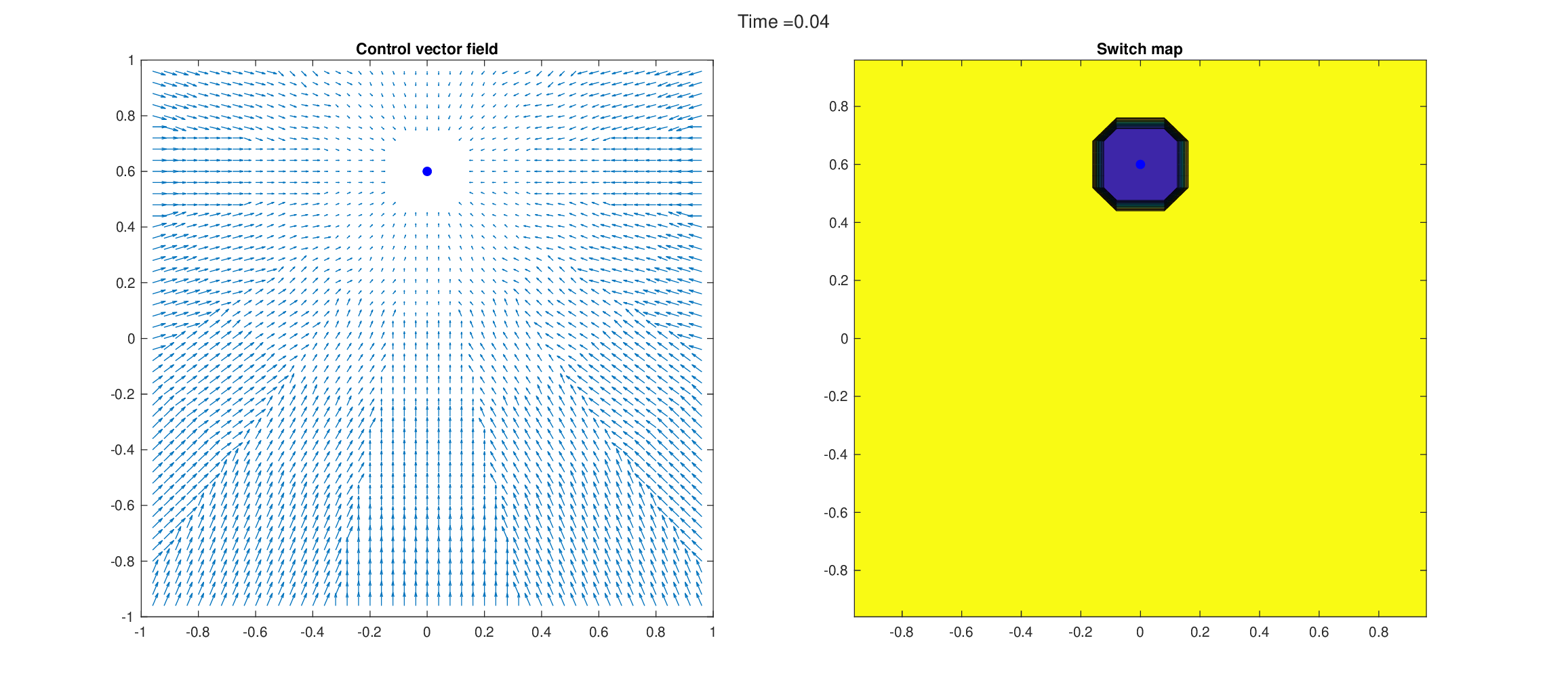}\\
\includegraphics[width=0.8\textwidth]{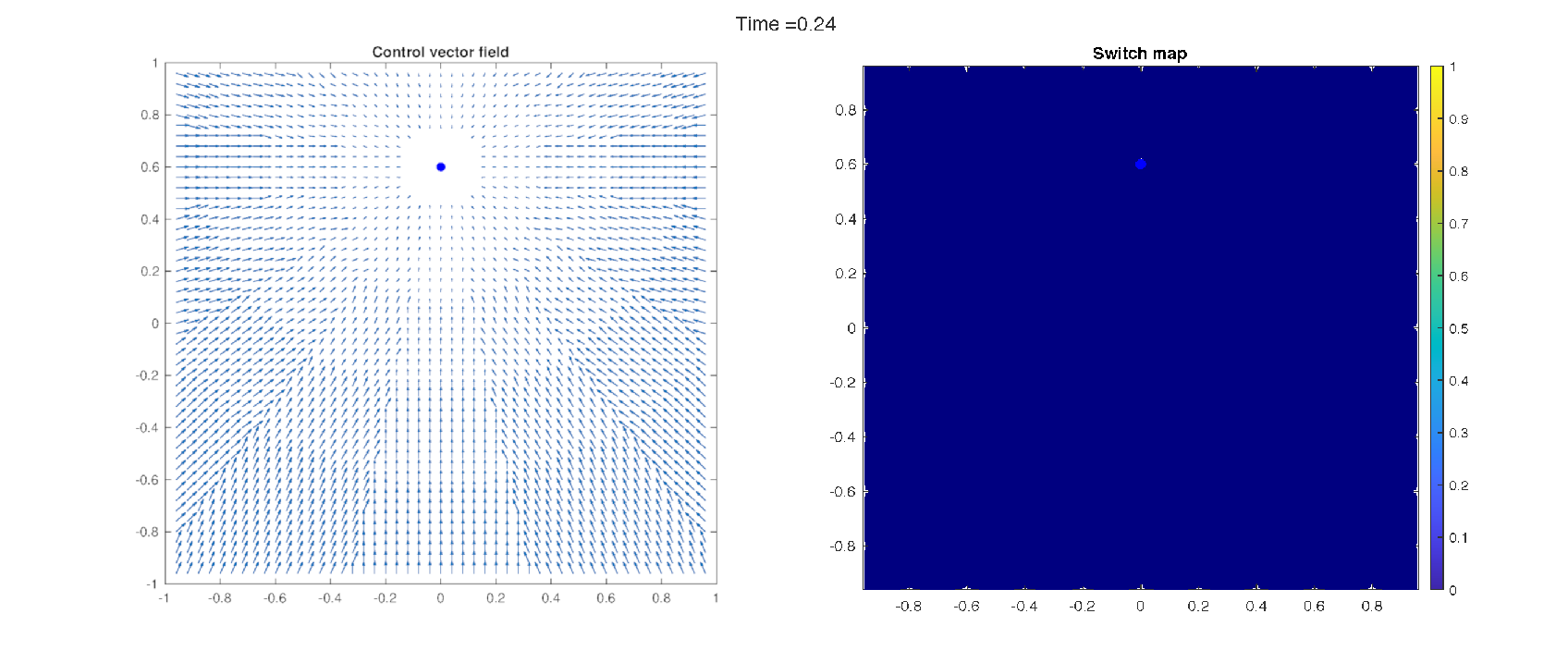}
\end{tabular}
\end{center}
\caption{Test 2. Control vector field (left) and switch map (right) of the optimal stopping problem at $t=0.04, 0.24$.}\label{4}
\end{figure}
\par\smallskip
As it was predictable in this easy scenario, we observe that the optimal solution suggests aiming directly to the target with speed inversely proportional to the square of the distance from it. Once one reached a region close enough to the target (approximately a ball of radius 0.2 as it is deducible from the switching map), it is more convenient to renounce to get closer to the target and paying the switching cost to end the game. Control vector field and switching map do not change during the evolution of the system until its final moment $t\approx T$: here, since it is more convenient paying the switching cost than the final boundary condition, the switching map includes all the domain. 

\begin{figure}[t]
\begin{center}
\begin{tabular}{c}

\includegraphics[width=0.95\textwidth]{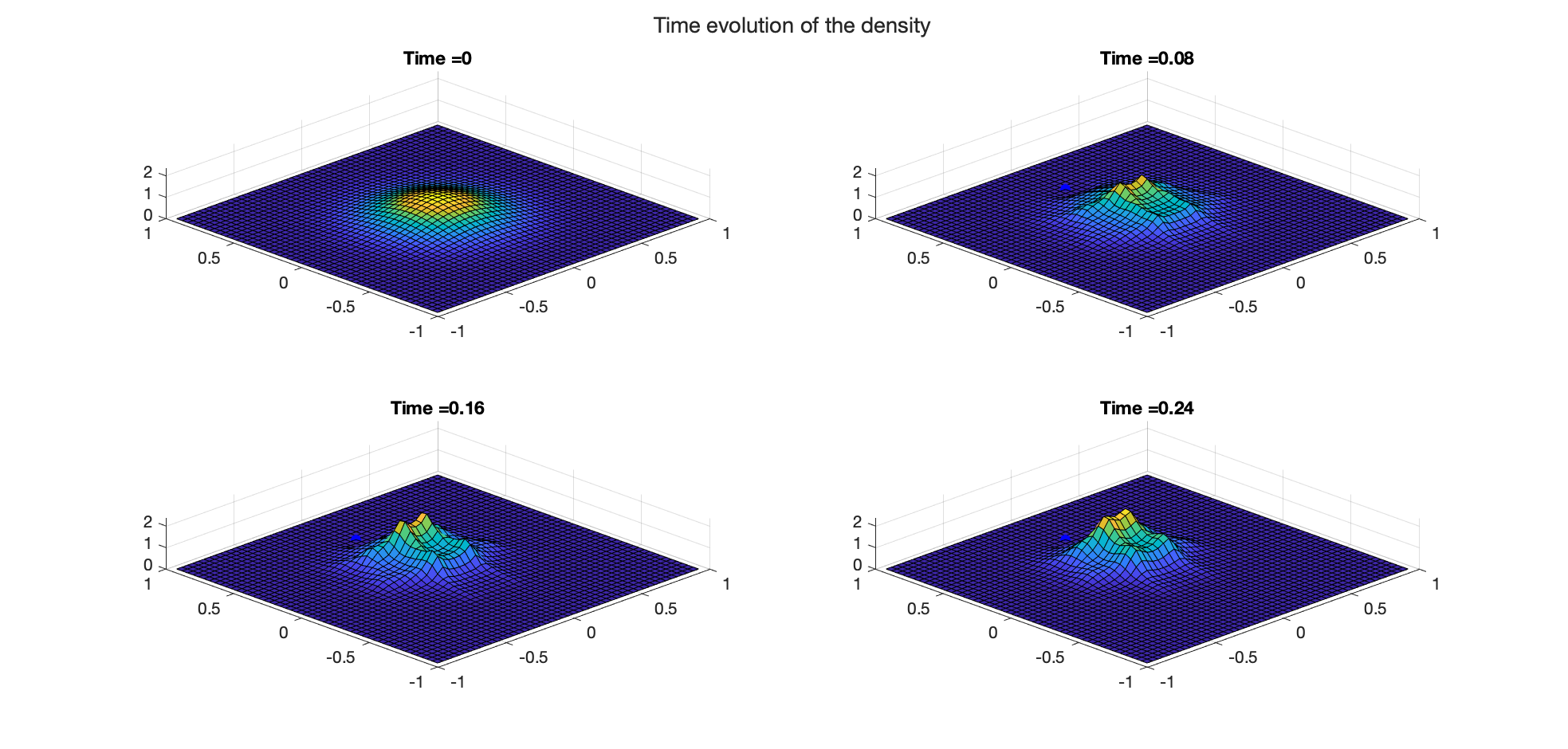}
\end{tabular}
\end{center}
\caption{Test 2. Density evolution at various instants $t=0, 0.06, 0.16, 0.24$. }\label{fig4}
\end{figure}
\par\smallskip
With this control vector field and switching map, we approximate the continuity equation \eqref{eq1} using the scheme \eqref{schemefp}. On the boundary, we use some standard homogeneous Neumann conditions as described in \cite{carlini2016PREPRINT} to preserve the total mass of agents. 
\par\smallskip
In Figure \ref{fig4}, we observe the evolution of the density. All agents directly aim to the target point in $(0,0.6)$ (in light blue) with a higher speed if more distant. This creates a specific concentration of the density in some areas of the domain close to where the switching map suggests renouncing to reach the target paying the switching cost. Clearly, inside this area, no density is present since it has already left the game. The density still present at $t=0.24$ will disappear at the next discrete time step since the switching map includes the whole domain. Therefore at the final step, $t=0.26$, the entire domain will be empty. 
\par\smallskip
We underline that in this case, since the control problem is independent of the distribution of mass, no congestion phenomena are considered, and any player acts as it is alone in the domain.

\subsection{The fully coupled case: two optimal visiting problems for a crowd of agents}

We add some complexity to our model. We repeat the previous test adding a dependency of \eqref{V} from the distribution of density $\mu(x,t)$. This makes our model a special kind of mean-field game. 

\medskip

We keep all the parameters of test 2 with the only exception of the final time $T=0.5$, $\Delta x=0.066$, $\Delta t=0.033$ and the running cost  
$$ \ell(x,a,p,t,\Upsilon)=e^{\sum\limits_{p\in \I} \Upsilon(x,t,p)}+\frac{\|a\|^2}{2}, \quad f(x,a,p)=a.$$
This choice is made to penalize the region of high density where, we suppose, the circulation is more difficult. All agents contributes to the congestion: even the ones that have switched already to another state of the system.
\par\smallskip
Due to the mutual dependency of the two equations \eqref{V}-\eqref{eq1}, we have to find a fixed point between the two equations. The existence of it has been largely discussed in various situations \cite{LasryLions06i,gomessurvey13} but never in the hybrid control framework. Some partial results are described in \cite{bertucci2018optimal} for an optimal stopping problem. To derive the fixed point numerically, the literature proposes two possible options: the use of a global descent method between the two-equation \cite{DY}, or a fixed point iteration between the two equations \cite{Carlini20154269,carlini2016PREPRINT}. We opted for the second one in this test, iterating between the two equations till a stopping criterium is verified. We use as stop criterium the distance in discrete infinite norm between two iterations of the density
$$ \cE(k)=\|\mu^z-\mu^{z-1}\|=\max_{k,i}|\mu^z(x_i, t_k)-\mu^{z-1}(x_i, t_k)|.$$
As we can see in Figure \ref{figstop}, after few iterations, the criterium $\cE(k)<\Delta x/2$ is reached.

\begin{figure}[t]
\begin{center}
\begin{tabular}{c}
\includegraphics[width=0.6 \textwidth]{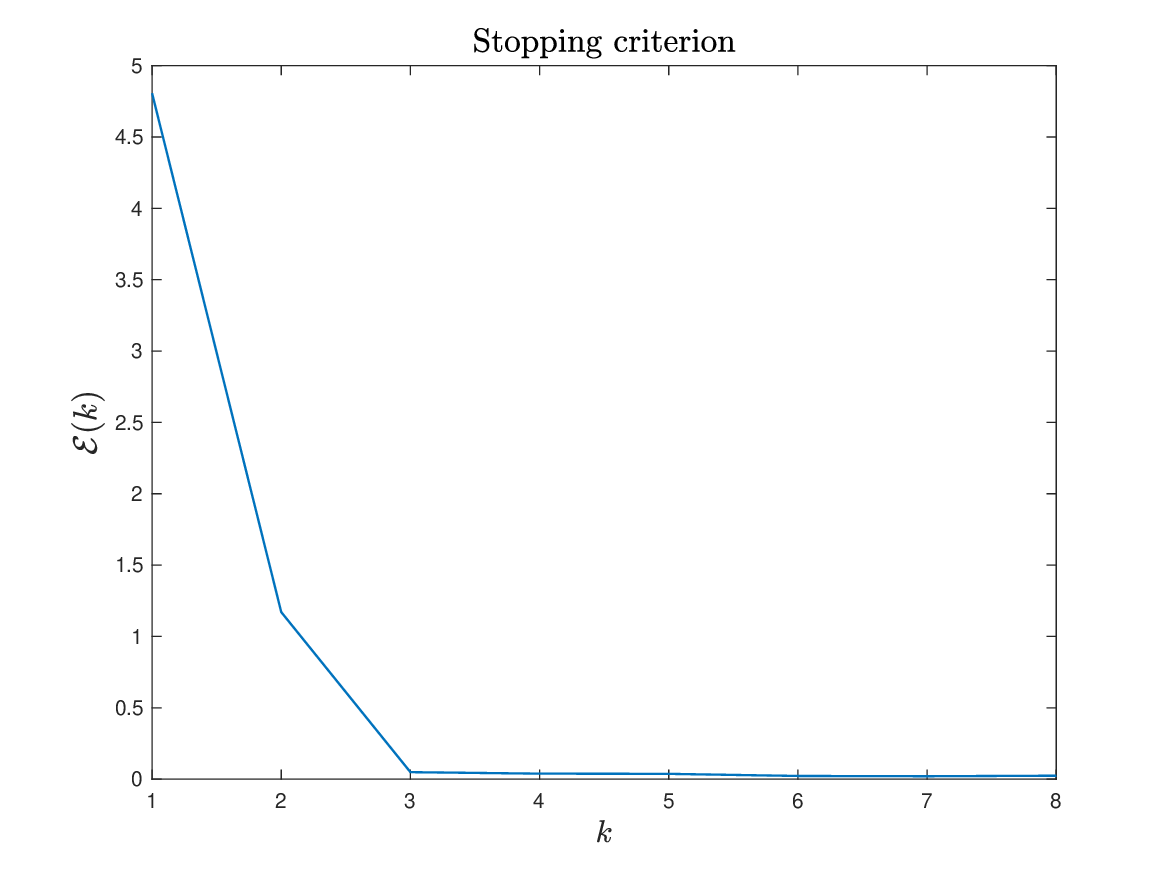}
\end{tabular}
\end{center}
\caption{Test 3. Evolution of the stopping criterium in the fixed point iteration.}\label{figstop}
\end{figure}
We show in Figure \ref{fig5} the control and switch map that we obtain. Different from Test 2, in this case, the control field is not constant but changes, trying actively to avoid the region of congestion. The latter is particularly clear at time $t=0.26$, where the agents around the center of the domain are encouraged to split laterally to avoid the congestion between them and the target. Also, the switching map is constant but moves in the congested areas, suggesting to pay the switching cost to the agents that would take too much time to circumnavigate the crowded regions. 

\begin{figure}[t]
\begin{center}
\begin{tabular}{c}
\includegraphics[width=10cm, height=4cm]{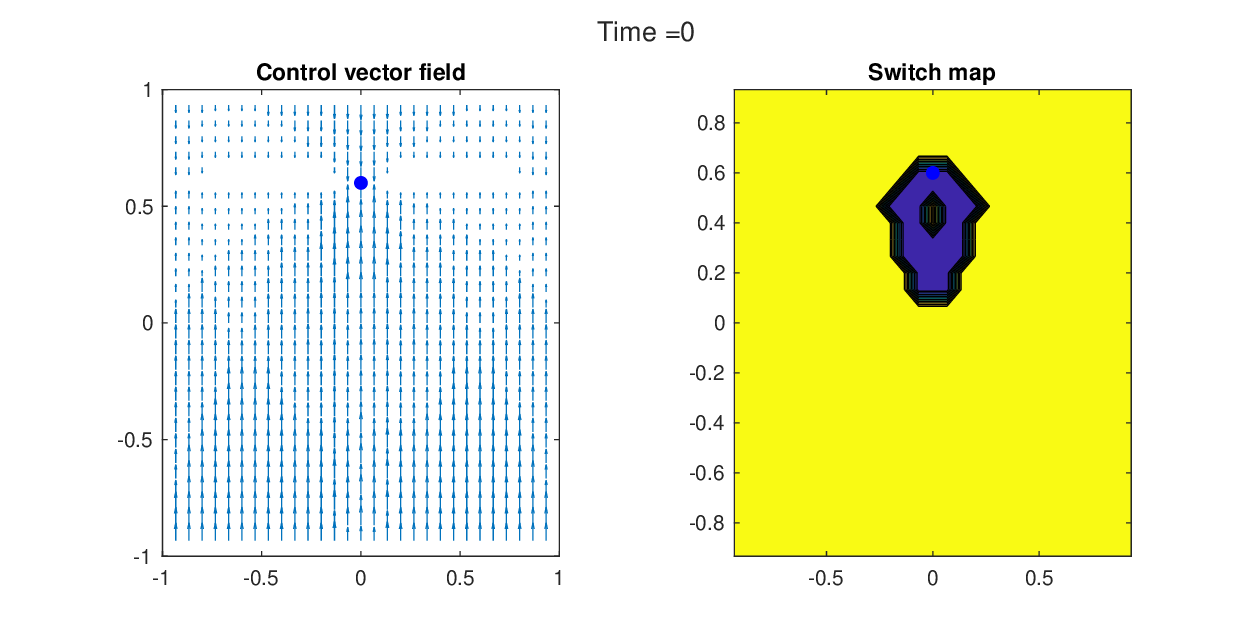}\\
\includegraphics[width=10cm, height=4cm]{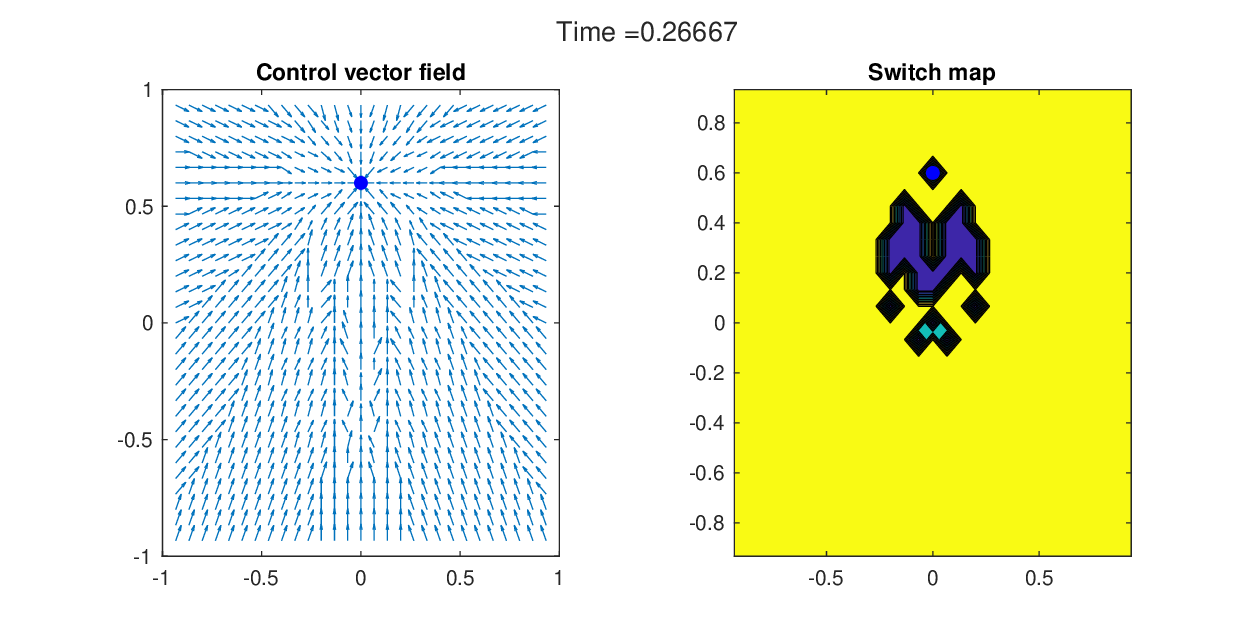}\\
\includegraphics[width=10cm, height=4cm]{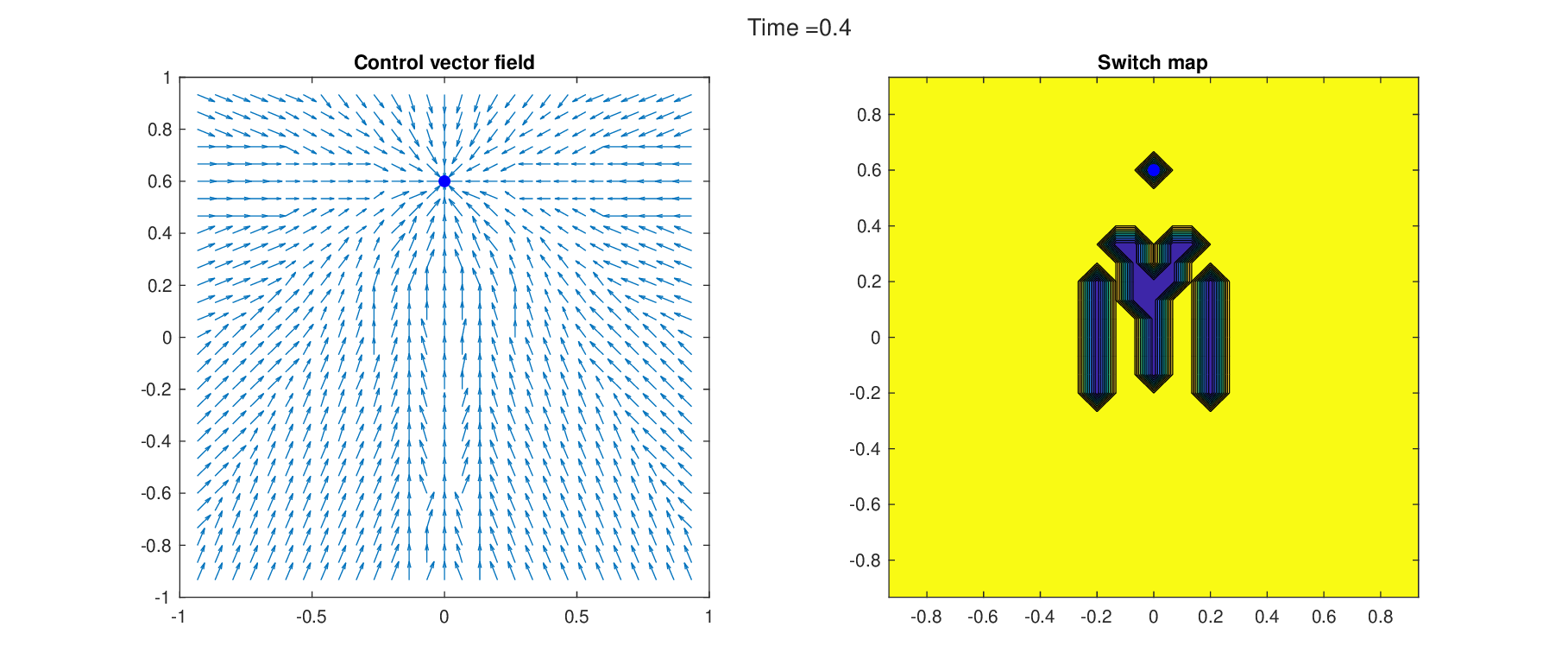}
\end{tabular}
\end{center}
\caption{Test 3. Control vector field (left) and switch map (right) of the optimal stopping problem at $t=0, 0.26, 0.4$. }\label{fig5}
\end{figure}
The same phenomenon is shown in Figure \ref{6}. In this case, we show the distribution of the density in state $(0)$ (target not visited) and $(1)$ (target visited). In the state $(1)$, as in all the final states where all targets are visited, the value function is always identically null, so it is the control map. For this reason, the density that switches to this state remains where the switch happened. Initially, at $t=0$, all the density is in state $(1)$; at $t=0.16$, only the portion of the density very close to the target does not consider it useful to get closer to the target has switched to state $(1)$. A part of the rest of the density creates a congested region in the target's direction, while the rest tries to avoid such a crowded area approaching the target laterally. This is evident at time t=0.32 where the density in $(1)$, i.e., the part of the density that has performed the switch, is all around the target point (differently from the previous Test 1). Finally, at $t=T=0.5$, all the density has switched to the final state. The state $0$ is empty, and state $(1)$ shows the location where all the switches happened. The density tends to aim to the target point spreading as much as possible to avoid congestions. A relevant part of the density switches almost immediately due to the penalization imposed in the congested areas.
\begin{figure}[t]
\begin{center}
\begin{tabular}{c}
\includegraphics[width=11cm, height=11cm]{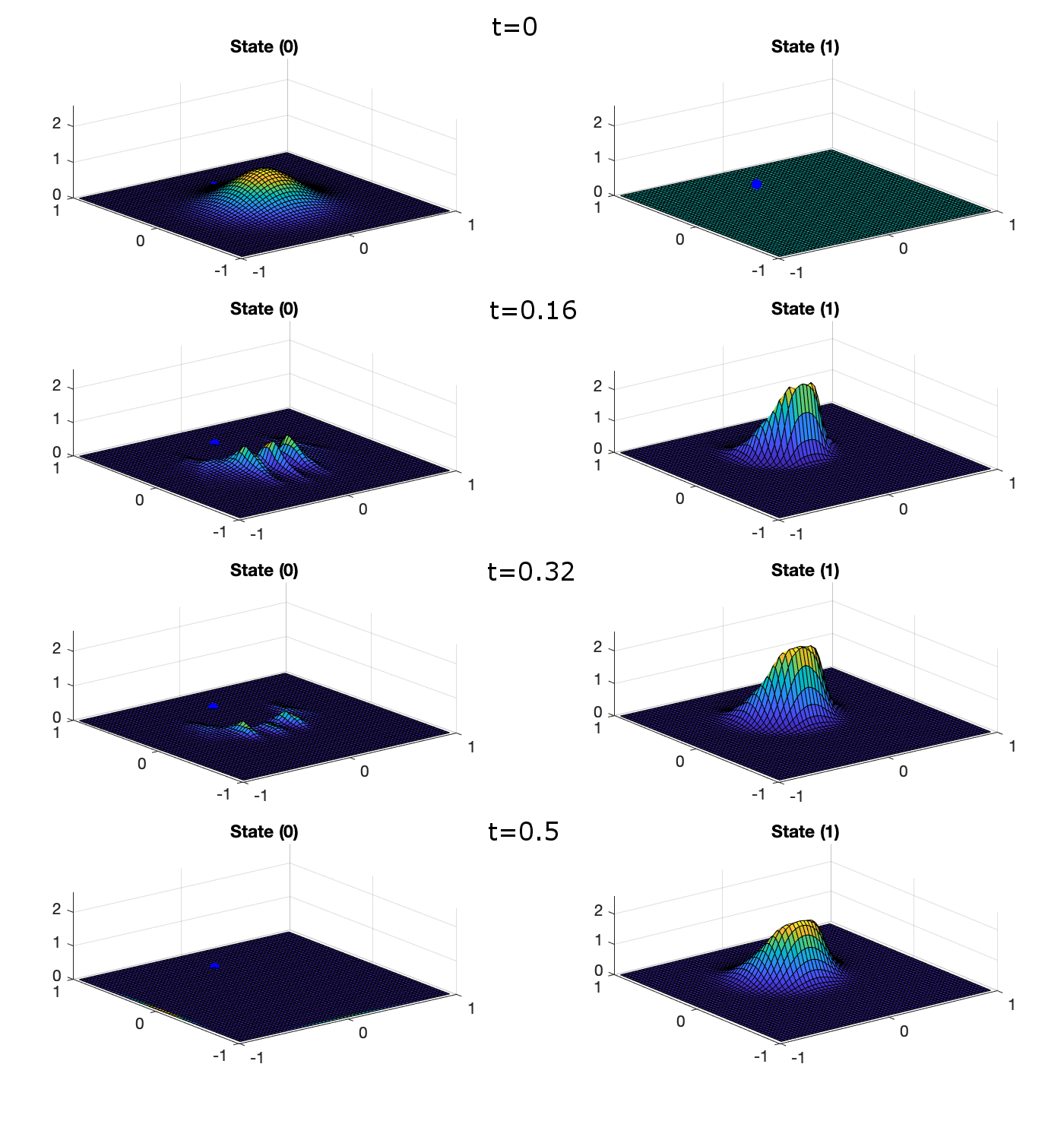}
\end{tabular}
\end{center}
\caption{Test 3. Density evolution at various instants $t=0, 0.16, 0.32, 0.5$.  }\label{6}
\end{figure}
\par\smallskip
Finally our last test is performed in the same situation of the single player Test 1. 
Therefore, the target set is composed by the points 
$$ \cT_j=0.6\,\left(\cos\left(j\frac{2\pi}{3}\right),\sin\left(j\frac{2\pi}{3}\right)\right), \quad j=1,\ldots,3.$$
The switching cost is again
$$ C(x,p,p')=\sum_{j\in\cI}\chi_j(p,p')\|x-\cT_j\|,$$
while the final time is $T=0.5$.  The discretization parameters are $\Delta x=0.06$ and $\Delta t=0.03$.
\par
Here the situation is more complex, and we limit our description to an overview of the behavior of the density reported in Figure \ref{7} and \ref{8}. We can see how the whole density, at time $t=0$, all in state $(0,0,0)$ splits in various parts aiming to the closer target (Figure \ref{7}). At $t=0.08$ some portions of the density have already switched a first time, appearing in the states $(1,0,0)$, $(0,1,0)$, $(0,0,1)$, while a part of the density located in the center of the domain (so then equidistant from the various targets) prefers to switch directly to the final state renouncing in this way to visit any of the targets. At this point, each one of these densities will aim (trying to avoid congestions) to the closest target point not visited yet. At $t=0.25$ (Figure \ref{8}) all the states (but the initial state $(0,0,0)$) contain a part of the density. This means, as we mentioned before, that for some agents, it was more convenient to renounce to visit one or even all the targets and paying the relative cost of the switch while in all the intermediate states, a portion of the density is still actively moving. Finally at time $t=0.5$ the whole density is in the final state $(1,1,1)$. As it was predictable, most of the switches happened close to the target points or in the region between them. All the intermediate states are empty.

\begin{figure}[t]
\begin{center}
\begin{tabular}{c}
\includegraphics[height=8cm]{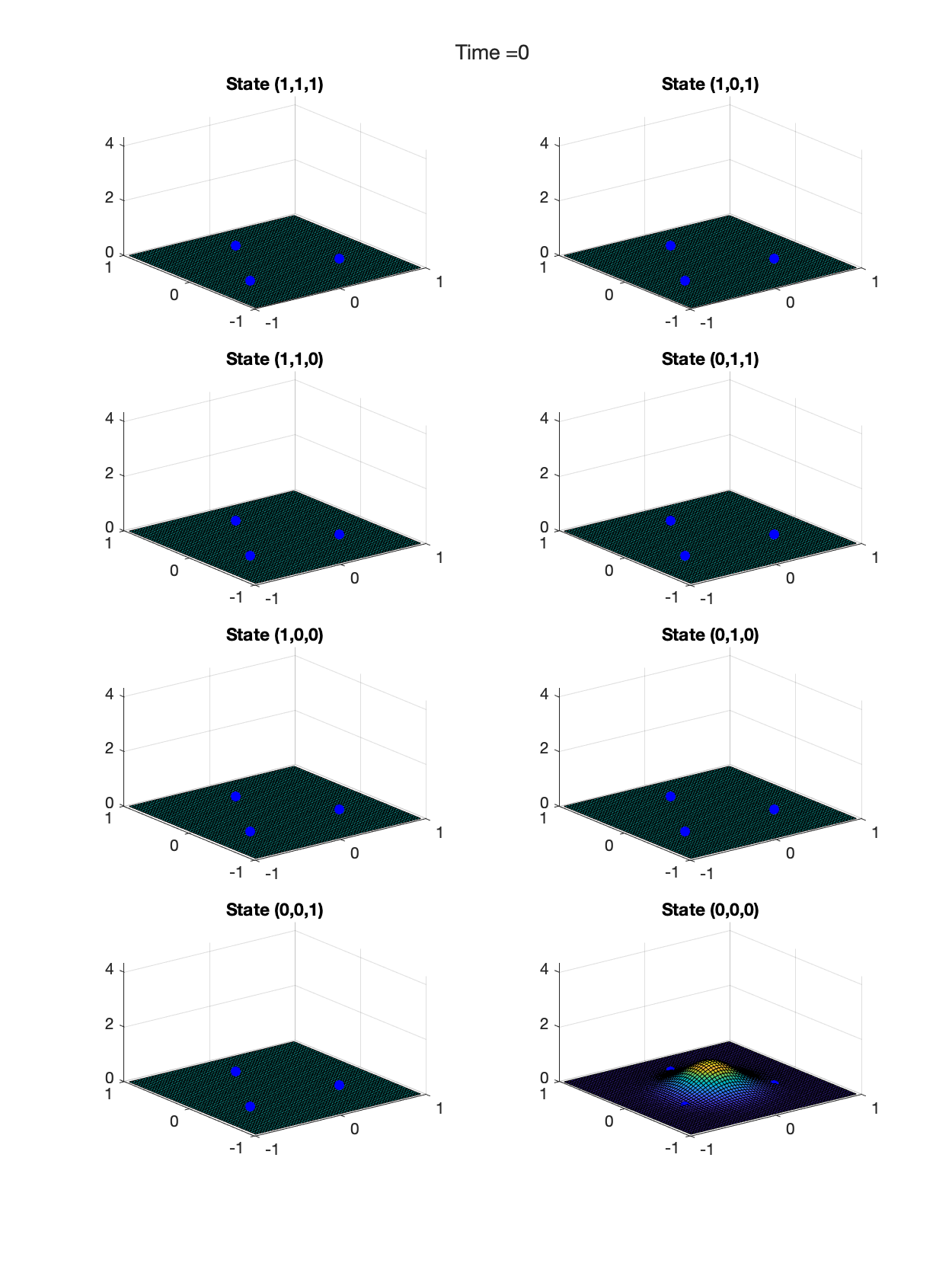}
\hspace{-0.8cm}
\includegraphics[height=8cm]{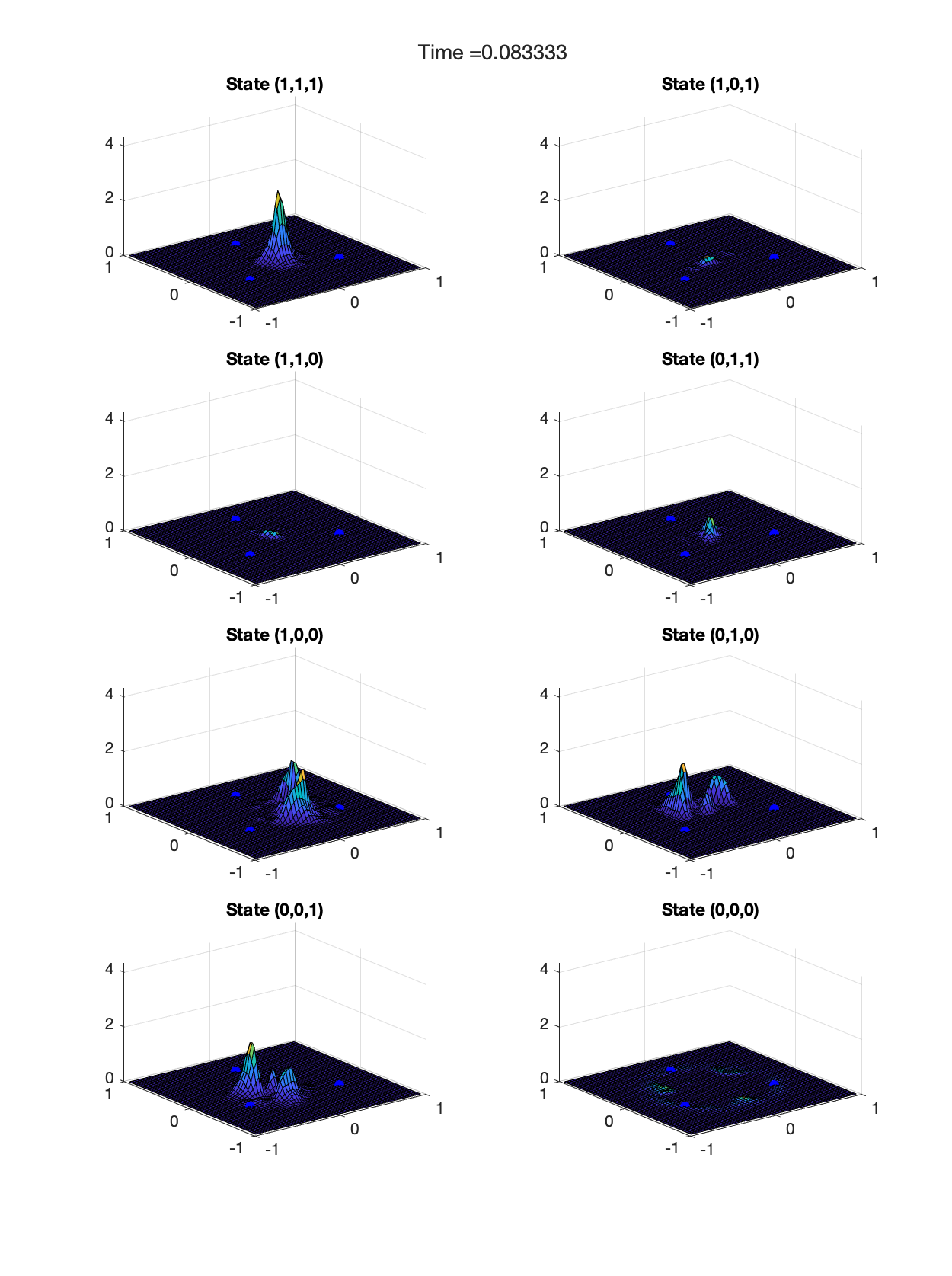}
\end{tabular}
\end{center}
\caption{Test 3. Approximated value functions in the various discrete states of the system}\label{7}
\end{figure}

\begin{figure}[t]
\begin{center}
\begin{tabular}{c}
\includegraphics[height=8cm]{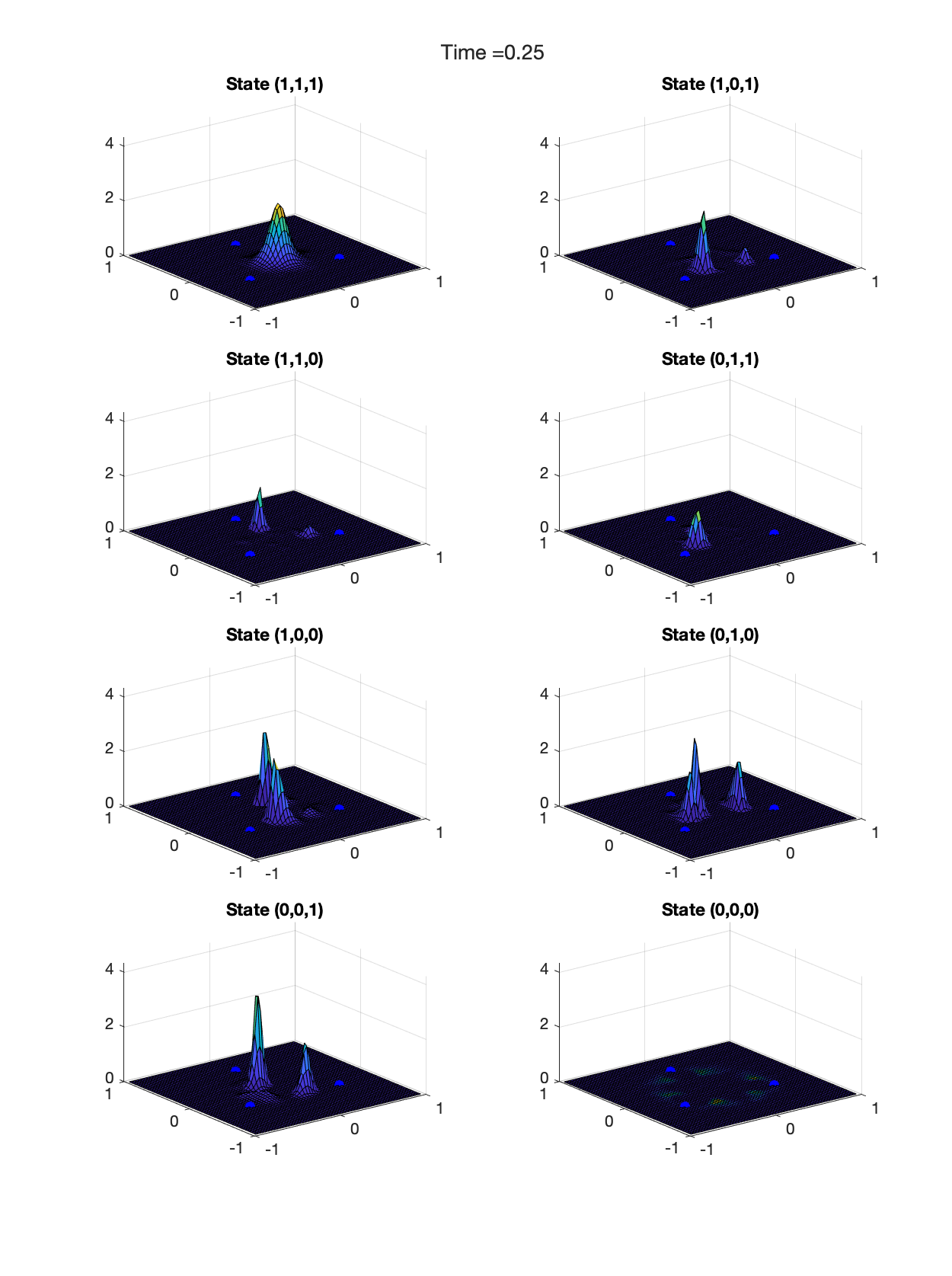}
\hspace{-0.8cm}
\includegraphics[height=8cm]{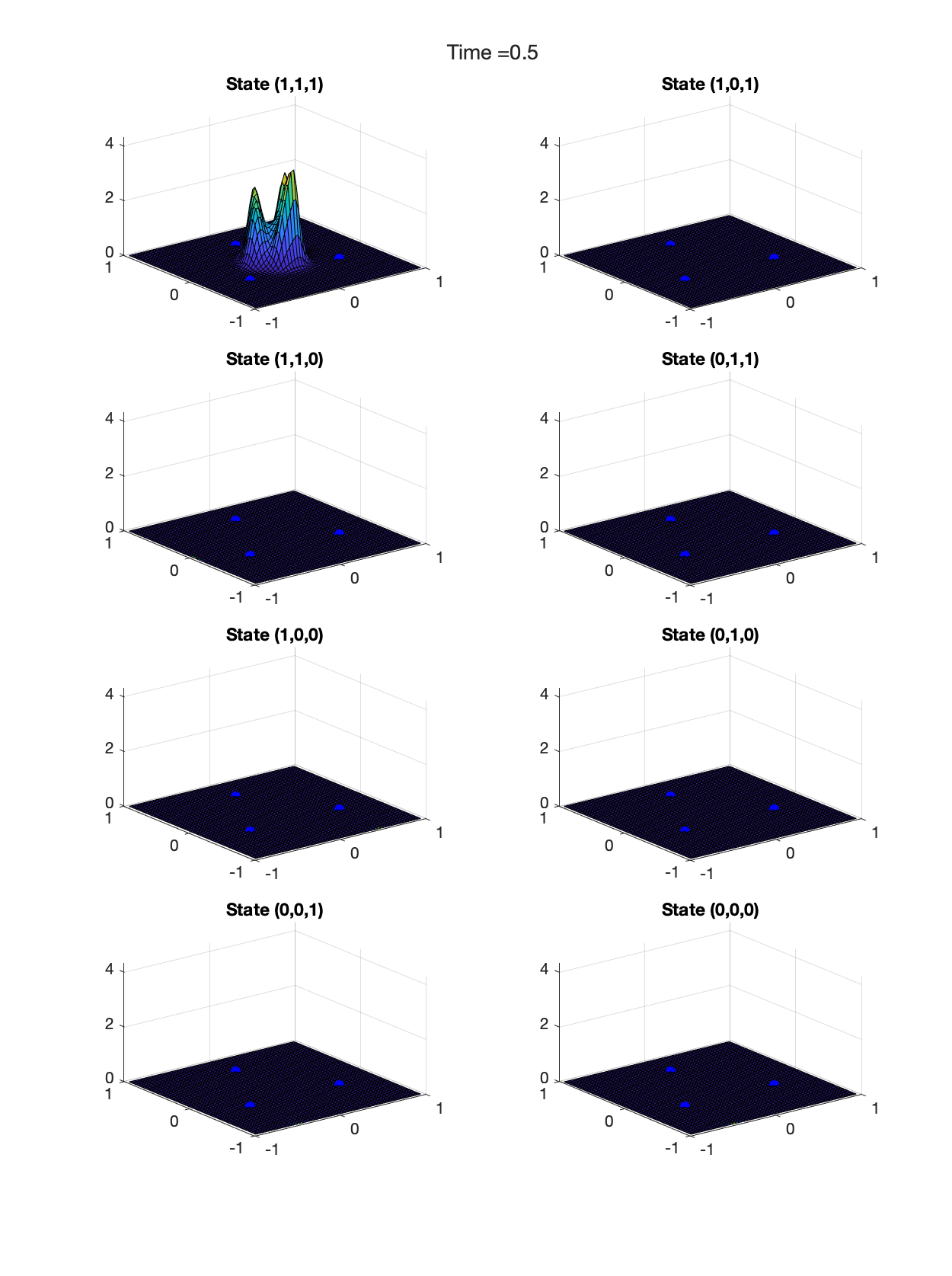}
\end{tabular}
\end{center}
\caption{Test 3. Approximated value functions in the various discrete states of the system}\label{8}
\end{figure}

\section{Conclusion}

The paper discusses both theoretical and applicative aspects of using hybrid control inside a possible macroscopic multi-agents system for an optimal visiting problem. The results that are shown here also constitute the first step toward a general theory of mean-field games in the presence of switches in the dynamics of the problem. This is an issue of considerable interest for hybrid-controlled systems but also for impulsive controlled structures. To complete the theory, various points must be fully developed: first of all, to consider, besides sinks, the presence of sources in the continuity equation. Afterwards, a major point would be the introduction of a real coupling between the continuity equation \eqref{eq1} and the HJ variational inequalities \eqref{V} via the optimal feedback (see Figure \ref{sch} too). In any case, as shown in the tests Section \S\ref{s-test2}, we may suppose, due to some promising numerical evidence, the existence of an equilibrium for such a coupling. We postpone this study to future research.

\end{document}